\documentclass[10pt]{amsart}
\usepackage{amssymb}
\usepackage{mathrsfs}
\usepackage{cite}
\usepackage{amscd}
\usepackage{amsmath}
\usepackage{algorithmicx}
\usepackage{algpseudocode}
\usepackage[all]{xy} 
\usepackage[usenames,dvipsnames]{color}
\usepackage{hyperref}

\theoremstyle{definition} 
\newtheorem{dfn}{Definition}
\newtheorem{ex}[dfn]{Example}
\newtheorem{pro}[dfn]{Proposition}
\newtheorem{lmm}[dfn]{Lemma}
\newtheorem*{thm*}{Theorem}
\newtheorem{thm}[dfn]{Theorem}
\newtheorem{cor}[dfn]{Corollary}

\newcommand{\proofPart}[1]{{\bf $\boldsymbol{#1}$:}}

\newcommand{\setBuilder}[2]{\left\{{#1}\left|{#2}\right.\right\}}
\newcommand{\setBuilderBarLeft}[2]{\left\{\left.{#1}\right|{#2}\right\}}
\newcommand{\setBuilderBarRight}[2]{\left\{{#1}\left|{#2}\right.\right\}}
\newcommand{\card}[1]{\left|#1\right|}
\newcommand{\p}{\ensuremath{^\prime}}
\newcommand{\pp}{\ensuremath{^{\prime\prime}}}
\newcommand{\ppp}{\ensuremath{^{\prime\prime\prime}}}
\newcommand{\set}[1]{\left\{{#1}\right\}}
\newcommand{\dsym}{\mathcal D}
\newcommand{\hsym}{\mathcal H}
\newcommand{\decom}[1]{\dsym\left({#1}\right)}
\newcommand{\decomn}[2]{\dsym_{#2}\left({#1}\right)}
\newcommand{\N}{\mathbb{N}}
\newcommand{\Z}{\mathbb{Z}}
\newcommand{\C}{\mathscr{C}}

\newcommand{\nodes}[1]{\mathcal{V}_{#1}}
\newcommand{\edges}[1]{\mathcal{E}_{#1}}
\newcommand{\labeling}[1]{\mathcal{L}_{#1}}
\newcommand{\lab}[2]{\labeling{#1}\left({#2}\right)}
\newcommand{\gadd}[2]{{#1}\oplus{#2}}
\newcommand{\gsub}[2]{{#1}\ominus{#2}}

\newcommand*{\defeq}{\mathrel{\vcenter{\baselineskip0.5ex \lineskiplimit0pt
                     \hbox{\scriptsize.}\hbox{\scriptsize.}}}%
                     =}
\newcommand*{\eqdef}{=
                     \mathrel{\vcenter{\baselineskip0.5ex \lineskiplimit0pt
                     \hbox{\scriptsize.}\hbox{\scriptsize.}}}%
                     }

\begin{document} 

\title{Connect Four and Graph Decomposition}
\author{Laurent Evain}
\email{laurent.evain@univ-angers.fr}
\address{Universit\'e d'Angers, 
Facult\'e des Sciences, 
D\'epartement de math\'ematiques,
2, Boulevard Lavoisier, 
49045 Angers Cedex 01, 
FRANCE}
\author{Mathias Lederer}
\email{mlederer@math.cornell.edu}
\address{Department of Mathematics, 
Cornell University, 
Malott Hall, 
Ithaca, 
New York 14853, 
USA}
\author{Bjarke Hammersholt Roune}
\email{bhroune@math.cornell.edu}
\address{Department of Mathematics, 
Cornell University, 
Malott Hall, 
Ithaca, 
New York 14853, 
USA}
\thanks{The second author was supported by a Marie Curie International
  Outgoing Fellowship of the EU Seventh Framework Program. The third
  author was supported by The Danish Council for Independent Research
  | Natural Sciences.}
\date{March 2013}
\keywords{Standard sets, labeled graphs, polynomial time complexity}
\subjclass[2000]{05A17; 05A19; 05C30; 05C78; 68Q25}

\begin{abstract} 
  We introduce standard decomposition, a natural way of decomposing a
  labeled graph into a sum of certain labeled subgraphs. We motivate
  this graph-theoretic concept by relating it to Connect Four
  decompositions of standard sets. We prove that all standard
  decompositions can be generated in polynomial time, which implies
  that all Connect Four decompositions can be generated in polynomial
  time.
\end{abstract}

\maketitle


\section{Introduction}

Let $G$ be a directed graph. We say that an integer-valued labeling on
the nodes of $G$ is \emph{compatible with the edge relation} if for
all edges $(a,b)$, the label of node $a$ is less than or equal to the label of
node $b$.  Graphs satisfying that compatibility form the class of
\emph{standard graphs}; they are the objects of study of the present
paper.

The paper is divided into two parts. In the first part, we study
standard graphs and introduce a way of decomposing a standard graph as
a sum of \emph{standard components} --- these are the standard
subgraphs of $G$ whose labels are 0 or 1. Here addition of labeled
graphs is defined as addition of the labels.  A \emph{standard
  decomposition} of a standard graph is a multiset of standard
components whose sum is the given graph.  Standard components may be
viewed as the building blocks of a standard graph.

Standard decomposition is not unique --- standard graphs in general
admit more than one standard decomposition.  Figure
\ref{figure:basicGraph} shows a simple example of a standard graph and
all its standard decompositions. This raises the question of what the
complexity of generating all standard decompositions given a standard
graph is. Theorem \ref{thm:generatingComplexity} answers this question
and it is the main result of the first part of the paper.

\begin{thm}
\phantomsection\label{thm:generatingComplexity}
  It is possible to generate all the standard
  decompositions of a standard graph in polynomial time.
\end{thm}

\begin{center}
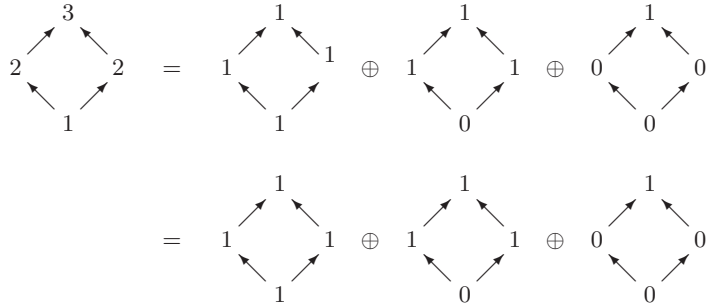
\begin{figure}[ht]
  \begin{picture}(260,125)
    \put(23,108){\small $3$}
    \put(10,95){\vector(1,1){10}}
    \put(3,87){\small $2$}
    \put(20,74){\vector(-1,1){10}}
    \put(30,74){\vector(1,1){10}}
    \put(23,66){\small $1$}
    \put(42,87){\small $2$}
    \put(40,95){\vector(-1,1){10}}
    \put(61,87){\small $=$}
    \put(103,108){\small $1$}
    \put(90,95){\vector(1,1){10}}
    \put(83,87){\small $1$}
    \put(100,74){\vector(-1,1){10}}
    \put(110,74){\vector(1,1){10}}
    \put(103,66){\small $1$}
    \put(122,92){\small $1$}
    \put(120,95){\vector(-1,1){10}}
    \put(136,87){\small $\oplus$}
    \put(173,108){\small $1$}
    \put(160,95){\vector(1,1){10}}
    \put(153,87){\small $1$}
    \put(170,74){\vector(-1,1){10}}
    \put(180,74){\vector(1,1){10}}
    \put(173,66){\small $0$}
    \put(192,87){\small $1$}
    \put(190,95){\vector(-1,1){10}}
    \put(206,87){\small $\oplus$}
    \put(243,108){\small $1$}
    \put(230,95){\vector(1,1){10}}
    \put(223,87){\small $0$}
    \put(240,74){\vector(-1,1){10}}
    \put(250,74){\vector(1,1){10}}
    \put(243,66){\small $0$}
    \put(262,87){\small $0$}
    \put(260,95){\vector(-1,1){10}}
    \put(61,22){\small $=$}
    \put(103,43){\small $1$}
    \put(90,30){\vector(1,1){10}}
    \put(83,22){\small $1$}
    \put(100,9){\vector(-1,1){10}}
    \put(110,9){\vector(1,1){10}}
    \put(103,1){\small $1$}
    \put(122,22){\small $1$}
    \put(120,30){\vector(-1,1){10}}
    \put(136,22){\small $\oplus$}
    \put(173,43){\small $1$}
    \put(160,30){\vector(1,1){10}}
    \put(153,22){\small $1$}
    \put(170,9){\vector(-1,1){10}}
    \put(180,9){\vector(1,1){10}}
    \put(173,1){\small $0$}
    \put(192,22){\small $1$}
    \put(190,30){\vector(-1,1){10}}
    \put(206,22){\small $\oplus$}
    \put(243,43){\small $1$}
    \put(230,30){\vector(1,1){10}}
    \put(223,22){\small $0$}
    \put(240,9){\vector(-1,1){10}}
    \put(250,9){\vector(1,1){10}}
    \put(243,1){\small $0$}
    \put(262,22){\small $0$}
    \put(260,30){\vector(-1,1){10}}
  \end{picture}
\caption{A standard graph and its two standard decompositions}
\phantomsection\label{figure:basicGraph}
\end{figure}
\end{center}

In the second part of the paper, we link standard graphs and standard decomposition to 
a previously studied subject --- \emph{Connect Four decomposition of standard sets}. 
A standard set is an ``$n$-dimensional staircase'', and a Connect
Four decomposition of a standard set $\Delta$ is a set of $n-1$ dimensional
standard sets from which $\Delta$ can be built by stacking them on top of each other 
and ``letting gravity pull them down''. 
Figure \ref{figure:basicC4} shows a simple example of a three
dimensional standard set and all its Connect Four decompositions; the
graph from Figure \ref{figure:basicGraph} encodes that same standard
set.

Connect Four decomposition is an ubiquitous notion  relevant to the study of
the Hilbert scheme of points \cite{components}. It is useful in the
study of singularities of plane curves as a tool for the Horace method
\cite{hirschowitz}, in the context of Gr\"obner basis theory \cite{eisenbud}, \cite{jpaa}, \cite{components}, to compute
tangent spaces \cite[Proposition 7.5]{nakajima}, or to produce new
counterexamples to Hilbert's fourteenth problem \cite{evain}. Handling Connect
Four decompositions is what
originally prompted the work in this paper.
We will show: 

\begin{thm}
\phantomsection\label{thm:graphsVsC4}
\begin{enumerate}
  \item[(i)] The two problems, 
    \begin{enumerate}
      \item[(a)] computing standard decompositions of labeled graphs, and
      \item[(b)] computing Connect Four decompositions of finite standard sets,
    \end{enumerate}
    are equivalent in the sense that for each labeled graph $G$, 
    there exists a standard set $\Delta$ such that the standard decompositions of $G$
    are in canonical bijection with the Connect Four decompositions of $\Delta$, and conversely. 
  \item[(ii)] This equivalence preserves polynomial complexity in the sense  
  that for each labeled graph $G$, we can compute a standard set $\Delta$ with graph $G$ in polynomial time, 
  and for each standard set $\Delta$, we can compute its graph $G(\Delta)$ in polynomial time. 
\end{enumerate}
\end{thm}

\begin{cor}
It is possible to generate all Connect Four decompositions of a
standard set in polynomial time.
\end{cor}

\begin{center}
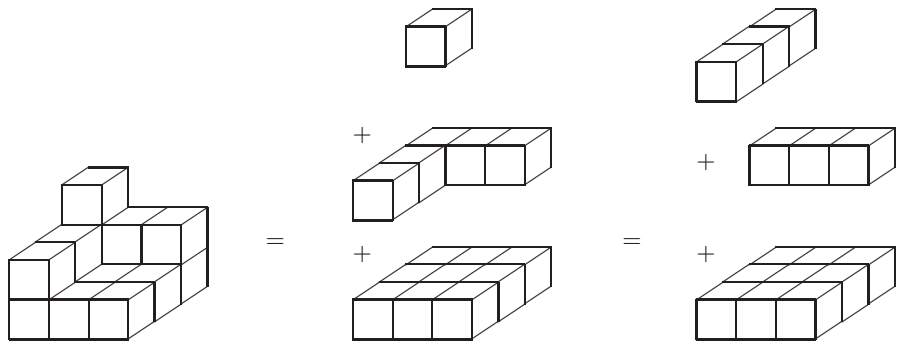
\begin{figure}[ht]
  \begin{picture}(335,140)
    \multiput(0,5)(0,15){2}{\line(1,0){45}}
    \multiput(0,35)(10,6.66){2}{\line(1,0){15}}
    \multiput(15,20)(15,0){2}{\line(3,2){20}}
    \put(15,35){\line(3,2){30}}
    \multiput(0,5)(15,0){2}{\line(0,1){30}}
    \multiput(30,5)(15,0){2}{\line(0,1){15}}
    \multiput(45,5)(0,15){2}{\line(3,2){30}}
    \multiput(75,25)(-10,-6.66){2}{\line(0,1){30}}
    \multiput(25,26.66)(10,6.66){2}{\line(1,0){30}}
    \put(25,26.66){\line(0,1){15}}
    \put(55,26.66){\line(0,-1){15}}
    \put(35,33.33){\line(0,1){30}}
    \put(50,33.33){\line(0,1){15}}
    \multiput(50,48.33)(15,0){2}{\line(3,2){10}}
    \put(65,48.33){\line(-1,0){45}}
    \put(75,55){\line(-1,0){30}}
    \put(0,35){\line(3,2){20}}
    \put(20,48.33){\line(0,1){15}}
    \multiput(20,63.33)(15,0){2}{\line(3,2){10}}
    \multiput(20,63.33)(10,6.66){2}{\line(1,0){15}}
    \put(45,55){\line(0,1){15}}
    \put(97,40){\small $=$}
    \put(130,35){\small $+$}
    \put(130,80){\small $+$}
    \put(232,40){\small $=$}
    \put(260,35){\small $+$}
    \put(260,70){\small $+$}
    \multiput(160,130)(-10,-6.66){2}{\line(1,0){15}}
    \multiput(160,130)(15,0){2}{\line(-3,-2){10}}
    \multiput(150,123.33)(15,0){2}{\line(0,-1){15}}
    \put(175,130){\line(0,-1){15}}
    \put(175,115){\line(-3,-2){10}}
    \put(150,108.33){\line(1,0){15}}
    \multiput(160,85)(15,0){2}{\line(-3,-2){30}}
    \multiput(160,85)(-10,-6.66){2}{\line(1,0){45}}
    \multiput(130,65)(10,6.66){2}{\line(1,0){15}}
    \multiput(145,50)(10,6.66){3}{\line(0,1){15}}
    \put(130,50){\line(0,1){15}}
    \put(130,50){\line(1,0){15}}
    \put(145,50){\line(3,2){20}}
    \multiput(205,85)(-15,0){2}{\line(-3,-2){10}}
    \put(205,70){\line(0,1){15}}
    \put(205,70){\line(-3,-2){10}}
    \multiput(195,63.33)(-15,0){2}{\line(0,1){15}}
    \put(195,63.33){\line(-1,0){30}}
    \put(130,5){\line(1,0){45}}
    \multiput(130,5)(15,0){3}{\line(0,1){15}}
    \multiput(130,20)(15,0){4}{\line(3,2){30}}
    \put(175,5){\line(3,2){30}}
    \multiput(175,5)(10,6.66){4}{\line(0,1){15}}
    \multiput(130,20)(10,6.66){4}{\line(1,0){45}}
    \multiput(290,130)(-10,-6.66){4}{\line(1,0){15}}
    \put(290,130){\line(-3,-2){30}}
    \multiput(305,130)(0,-15){2}{\line(-3,-2){30}}
    \multiput(305,130)(-10,-6.66){4}{\line(0,-1){15}}
    \put(260,110){\line(0,-1){15}}
    \put(260,95){\line(1,0){15}}
    \multiput(290,85)(-10,-6.66){2}{\line(1,0){45}}
    \put(280,63.33){\line(1,0){45}}
    \multiput(280,63.33)(15,0){3}{\line(0,1){15}}
    \multiput(325,63.33)(0,15){2}{\line(3,2){10}}
    \multiput(325,63.4)(10,6.66){2}{\line(0,1){15}}
    \multiput(280,78.33)(15,0){3}{\line(3,2){10}}    
    \put(260,5){\line(1,0){45}}
    \multiput(260,5)(15,0){3}{\line(0,1){15}}
    \multiput(260,20)(15,0){4}{\line(3,2){30}}
    \put(305,5){\line(3,2){30}}
    \multiput(305,5)(10,6.66){4}{\line(0,1){15}}
    \multiput(260,20)(10,6.66){4}{\line(1,0){45}}
  \end{picture}
\caption{A standard set and its two Connect Four decompositions}
\phantomsection\label{figure:basicC4}
\end{figure}
\end{center}

We conclude our paper with an appendix which links the notions
introduced in this paper to other classical tools and problems. 
First we present a generating function for the number of standard
decompositions of a given graph. 
Then we show that the set of all \emph{Connect Four games} in $\N^d$ of a given size $n$ is in canonical
bijection with the set of $(d-1)$-fold iterated partitions of $n$.

\subsubsection*{A word on the proofs}
\phantomsection\label{sec:word-proofs}

We prove Theorem \ref{thm:generatingComplexity} by presenting an
algorithm that generates all standard decompositions in polynomial
time.  The algorithm is based on reducing the problem of computing all
standard decompositions of $G$ to the problem of computing all
standard decompositions of $G$ containing a fixed node $v$.  We then
solve that problem in a recursive way. Any choice of the node $v$
results in a correct algorithm, yet we give a specific choice of $v$
that allows the algorithm to generate its output in polynomial time.

The proof of Theorem \ref{thm:graphsVsC4} is done in several steps. 
We first attach a graph $G(\Delta)$ to each standard set $\Delta$ such that the standard
decompositions of $G(\Delta)$ and the Connect Four decompositions of
$\Delta$ are in canonical bijection. From $G(\Delta)$ we then define
another graph $G\p(\Delta)$ that is easier to work with, called the
\emph{canonicalized standard graph}, which has the same
decompositions (see Proposition \ref{pro:canonicalization}). We show
that all labeled graphs arising from standard sets in this way have
three specific properties, namely,
\begin{itemize}
  \item they are standard,
  \item they are connected, and 
  \item they have a unique node of maximal label. 
\end{itemize}
Let $\mathcal{S}$ be the class of labeled graphs satisfying these
conditions. The connectedness assumption in the definition of
$\mathcal{S}$ is not essential for the complexity of the graphs from
that class, since the standard decompositions of a disjoint union of
graphs is the product of the standard decompositions of the individual
graphs. We prove in Proposition \ref{pro:graphToStdSet} that each
connected graph in $\mathcal{S}$ arises from a standard set if, in
addition, the relation on the nodes of the graph defined by the edges
of the graph is transitive. In Proposition
\ref{pro:reductionToUniqueMaximalNode}, 
we show that for each
connected standard graph, there exists a graph in $\mathcal{S}$ such
that the standard decompositions of the two graphs are in canonical bijection. 

\subsection*{Acknowledgements}

Thanks to Dustin Cartwright, Daniel Erman, Allen Knutson, Diane
Maclagan, Jenna Rajchgot, Roy Skjelnes, Greg Smith, Mike Stillman
and Elmar Teufl for fruitful discussions.


\section{Standard graphs and standard components}
\phantomsection\label{sec:standardGraphsAndComponents}

All graphs under consideration are directed, have finitely many nodes
and do not have any parallel edges or loops. Given a graph, let $<$ be
the partial preorder on the set of nodes such that $a < b$ if $b$ is
reachable from $a$. The graphs that we consider are labeled in the
following sense.

\begin{dfn}[Labeled graph]
A \emph{labeled graph} is a graph $G$ with a finite node set
$\nodes G$, an edge set $\edges G\subseteq\nodes G\times\nodes G$
such that the graph contains no loops, 
and a labeling of nodes $\labeling G\colon \nodes G\rightarrow\Z$. 
\end{dfn}

This definition does not allow parallel edges since the edge set is
not a multiset. The constraints on parallel edges and loops are not
important to the results of this paper. We impose those conditions for
simplicity since loops and parallel edges add nothing interesting to
the problem.

\begin{dfn}[Standard graph]
A labeled graph $G$ is \emph{standard} if all labels are non-negative
and the labeling is compatible with the partial order on the nodes in the sense that 
$\lab G a \leq \lab G b$ for all edges $(a,b)\in\edges G$.
\end{dfn}

We now introduce the operations of \emph{addition} and
\emph{subtraction} on labeled graphs.

\begin{dfn}[Addition and subtraction]
\phantomsection\label{dfn:additionSubtraction}
Let $G$ and $H$ be labeled graphs. 
Then $\gadd G H$ 
is the labeled graph with node set $\nodes {\gadd G H} \defeq {\nodes G} \cup {\nodes H}$, 
edge set $\edges {\gadd G H} \defeq {\edges G} \cup {\edges H}$
and labeling
\[
\labeling{\gadd G H}\defeq
\begin{cases}
\labeling G&\text{ for }v\in\nodes G\setminus\nodes H,\\
\labeling G + \labeling H & \text{ for }v\in\nodes G\cap\nodes H,\\
\labeling H&\text{ for }v\in\nodes H\setminus\nodes G.
\end{cases}
\]
We define $\gsub G H$ to have the same node set and edge set as $\gadd G H$, 
but with labeling
\[
\labeling{\gsub G H}\defeq
\begin{cases}
\labeling G&\text{ for }v\in\nodes G\setminus\nodes H,\\
\labeling G - \labeling H & \text{ for }v\in\nodes G\cap\nodes H,\\
-\labeling H&\text{ for }v\in\nodes H\setminus\nodes G.
\end{cases}
\]
\end{dfn}

 The sum of two standard graphs \emph{with the same set of nodes and
   the same set of edges} is again a standard graph. This is in
 general true when the set of nodes or edges differ, as shown in the
 example from Figure \ref{figure:sumOfStdIsntStd}.  In the present
 paper, we will only consider sums and differences of graphs sharing
 the same set of nodes and edges.

\begin{center}
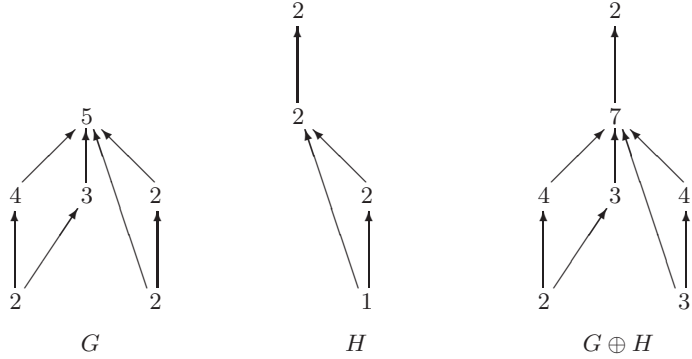
\begin{figure}[ht]
  \begin{picture}(260,130)
    \put(30,86){\small $5$}
    \put(8,64){\vector(1,1){20}}
    \put(32,64){\vector(0,1){20}}
    \put(58,64){\vector(-1,1){20}}
    \put(3,56){\small $4$}
    \put(30,56){\small $3$}
    \put(56,56){\small $2$}
    \put(55,24){\vector(-1,3){20}}
    \put(59,24){\vector(0,1){29}}
    \put(9,24){\vector(2,3){20}}
    \put(5,24){\vector(0,1){29}}
    \put(3,16){\small $2$}
    \put(56,16){\small $2$}
    \put(30,0){\small $G$}    
    \put(110,126){\small $2$}
    \put(112,94){\vector(0,1){29}}
    \put(110,86){\small $2$}
    \put(138,64){\vector(-1,1){20}}
    \put(136,56){\small $2$}
    \put(135,24){\vector(-1,3){20}}
    \put(139,24){\vector(0,1){29}}
    \put(136,16){\small $1$}
    \put(130,0){\small $H$}
    \put(230,126){\small $2$}
    \put(232,94){\vector(0,1){29}}
    \put(230,86){\small $7$}
    \put(208,64){\vector(1,1){20}}
    \put(232,64){\vector(0,1){20}}
    \put(258,64){\vector(-1,1){20}}
    \put(203,56){\small $4$}
    \put(230,56){\small $3$}
    \put(256,56){\small $4$}
    \put(255,24){\vector(-1,3){20}}
    \put(259,24){\vector(0,1){29}}
    \put(209,24){\vector(2,3){20}}
    \put(205,24){\vector(0,1){29}}
    \put(203,16){\small $2$}
    \put(256,16){\small $3$}
    \put(220,0){\small $\gadd G H$}    
  \end{picture}
\caption{The sum of two standard graphs}
\phantomsection\label{figure:sumOfStdIsntStd}
\end{figure}
\end{center}

\begin{dfn}[0-1 graph]
A labeled graph is a \emph{0-1 graph} if all labels are 0 or 1.
\end{dfn}

If we take a standard graph and replace all positive labels by 1, then
we obtain another standard graph. This is a \emph{standard 0-1 graph}
--- a graph that is both standard and a 0-1 graph. Some subgraphs $H$
of a standard graph $G$ are standard 0-1 graphs and in some cases we
can write $G$ as $\gadd H G\p$, where $G\p$ is another standard
graph. In this case we call $H$ a \emph{standard component} of $G$.

\begin{dfn}[Standard component]
Let $G$ and $H$ be labeled graphs. 
Then $H$ is a \emph{standard component} of $G$ if
\begin{enumerate}
\item $H$ is a standard 0-1 graph;
\item $\gsub G H$ is a standard graph; and 
\item not all labels in $H$ are zero.
\end{enumerate}
\end{dfn}

We think of standard components of $G$ as the building blocks of $G$.  Our goal is
to determine all the ways to build a graph out of such building
blocks.

\begin{dfn}[Standard decomposition]
Let $G$ be a labeled graph. A multiset of labeled graphs $\hsym$ is a
\emph{standard decomposition} of $G$ if each $H\in\hsym$ is a standard
component of $G$ and $G=\sum_{H \in \hsym} H$. We denote the set of
standard decompositions of $G$ by $\decom G$.
\end{dfn}

To keep formulas succint, we use the shorthand notation
$\sum\hsym\defeq\sum_{H \in \hsym} H$.

A standard 0-1 graph $G$ admits only the standard decomposition
$\{G\}$.  In particular, the building blocks of a graph are
indecomposable; this is why we call them standard \emph{components}.
We define standard decompositions to be multisets rather than sets
since a standard component can appear multiple times within one
decomposition.

\begin{ex}\phantomsection\label{ex:}
  Figure \ref{figure:decompositions} shows a standard graph and all
  its decompositions. 
\end{ex}

\begin{center}
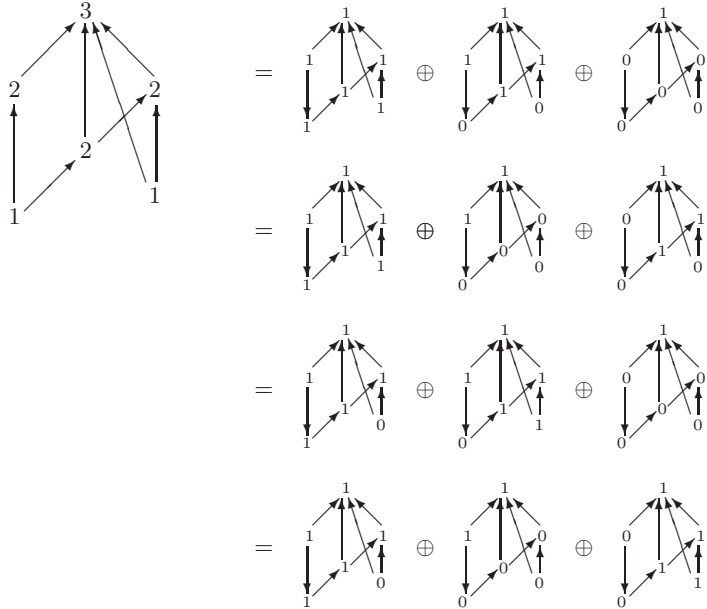
\begin{figure}[ht]
  \begin{picture}(290,230)
    \put(30,226){\small $3$}
    \put(8,204){\vector(1,1){20}}
    \put(32,181){\vector(0,1){43}}
    \put(58,204){\vector(-1,1){20}}
    \put(3,196){\small $2$}
    \put(30,173){\small $2$}
    \put(56,196){\small $2$}
    \put(55,164){\vector(-1,3){20}}
    \put(59,164){\vector(0,1){29}}
    \put(9,153){\vector(1,1){19}}
    \put(5,156){\vector(0,1){37}}
    \put(37,179){\vector(1,1){18}}
    \put(3,148){\small $1$}
    \put(56,156){\small $1$}
    \put(96,203){\small $=$}    
    \put(129,226){\tiny $1$}
    \put(115,208){\tiny $1$}
    \put(128.5,196){\tiny $1$}
    \put(143,208){\tiny $1$}
    \put(114,183){\tiny $1$}
    \put(142,190){\tiny $1$}
    \put(118,214){\vector(1,1){10}}
    \put(129,201){\vector(0,1){23}}
    \put(141,196){\vector(-1,3){9.5}}
    \put(144,214){\vector(-1,1){10}}
    \put(116,206){\vector(0,-1){18}}
    \put(118,187){\vector(1,1){10}}
    \put(144,196){\vector(0,1){10}}
    \put(132.5,199){\vector(1,1){10}}
    \put(157,203){\small $\oplus$}    
    \put(189,226){\tiny $1$}
    \put(175,208){\tiny $1$}
    \put(188.5,196){\tiny $1$}
    \put(203,208){\tiny $1$}
    \put(173,183){\tiny $0$}
    \put(202,190){\tiny $0$}
    \put(217,203){\small $\oplus$}    
    \put(178,214){\vector(1,1){10}}
    \put(189,201){\vector(0,1){23}}
    \put(201,196){\vector(-1,3){9.5}}
    \put(204,214){\vector(-1,1){10}}
    \put(176,206){\vector(0,-1){18}}
    \put(178,187){\vector(1,1){10}}
    \put(204,196){\vector(0,1){10}}
    \put(192.5,199){\vector(1,1){10}}
    \put(249,226){\tiny $1$}
    \put(235,208){\tiny $0$}
    \put(248.5,196){\tiny $0$}
    \put(263,208){\tiny $0$}
    \put(233,183){\tiny $0$}
    \put(262,190){\tiny $0$}
    \put(238,214){\vector(1,1){10}}
    \put(249,201){\vector(0,1){23}}
    \put(261,196){\vector(-1,3){9.5}}
    \put(264,214){\vector(-1,1){10}}
    \put(236,206){\vector(0,-1){18}}
    \put(238,187){\vector(1,1){10}}
    \put(264,196){\vector(0,1){10}}
    \put(252.5,199){\vector(1,1){10}}
    \put(96,143){\small $=$}    
    \put(129,166){\tiny $1$}
    \put(115,148){\tiny $1$}
    \put(128.5,136){\tiny $1$}
    \put(143,148){\tiny $1$}
    \put(114,123){\tiny $1$}
    \put(142,130){\tiny $1$}
    \put(157,143){\small $\oplus$}
    \put(118,154){\vector(1,1){10}}
    \put(129,141){\vector(0,1){23}}
    \put(141,136){\vector(-1,3){9.5}}
    \put(144,154){\vector(-1,1){10}}
    \put(116,146){\vector(0,-1){18}}
    \put(118,127){\vector(1,1){10}}
    \put(144,136){\vector(0,1){10}}
    \put(132.5,139){\vector(1,1){10}}
    \put(157,143){\small $\oplus$}    
    \put(189,166){\tiny $1$}
    \put(175,148){\tiny $1$}
    \put(188.5,136){\tiny $0$}
    \put(203,148){\tiny $0$}
    \put(173,123){\tiny $0$}
    \put(202,130){\tiny $0$}
    \put(178,154){\vector(1,1){10}}
    \put(189,141){\vector(0,1){23}}
    \put(201,136){\vector(-1,3){9.5}}
    \put(204,154){\vector(-1,1){10}}
    \put(176,146){\vector(0,-1){18}}
    \put(178,127){\vector(1,1){10}}
    \put(204,136){\vector(0,1){10}}
    \put(192.5,139){\vector(1,1){10}}
    \put(217,143){\small $\oplus$}    
    \put(249,166){\tiny $1$}
    \put(235,148){\tiny $0$}
    \put(248.5,136){\tiny $1$}
    \put(263,148){\tiny $1$}
    \put(233,123){\tiny $0$}
    \put(262,130){\tiny $0$}
    \put(238,154){\vector(1,1){10}}
    \put(249,141){\vector(0,1){23}}
    \put(261,136){\vector(-1,3){9.5}}
    \put(264,154){\vector(-1,1){10}}
    \put(236,146){\vector(0,-1){18}}
    \put(238,127){\vector(1,1){10}}
    \put(264,136){\vector(0,1){10}}
    \put(252.5,139){\vector(1,1){10}}
    \put(96,83){\small $=$}    
    \put(129,106){\tiny $1$}
    \put(115,88){\tiny $1$}
    \put(128.5,76){\tiny $1$}
    \put(143,88){\tiny $1$}
    \put(114,63){\tiny $1$}
    \put(142,70){\tiny $0$}
    \put(118,94){\vector(1,1){10}}
    \put(129,81){\vector(0,1){23}}
    \put(141,76){\vector(-1,3){9.5}}
    \put(144,94){\vector(-1,1){10}}
    \put(116,86){\vector(0,-1){18}}
    \put(118,67){\vector(1,1){10}}
    \put(144,76){\vector(0,1){10}}
    \put(132.5,79){\vector(1,1){10}}
    \put(157,83){\small $\oplus$}    
    \put(189,106){\tiny $1$}
    \put(175,88){\tiny $1$}
    \put(188.5,76){\tiny $1$}
    \put(203,88){\tiny $1$}
    \put(173,63){\tiny $0$}
    \put(202,70){\tiny $1$}
    \put(178,94){\vector(1,1){10}}
    \put(189,81){\vector(0,1){23}}
    \put(201,76){\vector(-1,3){9.5}}
    \put(204,94){\vector(-1,1){10}}
    \put(176,86){\vector(0,-1){18}}
    \put(178,67){\vector(1,1){10}}
    \put(204,76){\vector(0,1){10}}
    \put(192.5,79){\vector(1,1){10}}
    \put(217,83){\small $\oplus$}    
    \put(249,106){\tiny $1$}
    \put(235,88){\tiny $0$}
    \put(248.5,76){\tiny $0$}
    \put(263,88){\tiny $0$}
    \put(233,63){\tiny $0$}
    \put(262,70){\tiny $0$}
    \put(238,94){\vector(1,1){10}}
    \put(249,81){\vector(0,1){23}}
    \put(261,76){\vector(-1,3){9.5}}
    \put(264,94){\vector(-1,1){10}}
    \put(236,86){\vector(0,-1){18}}
    \put(238,67){\vector(1,1){10}}
    \put(264,76){\vector(0,1){10}}
    \put(252.5,79){\vector(1,1){10}}
    \put(96,23){\small $=$}    
    \put(129,46){\tiny $1$}
    \put(115,28){\tiny $1$}
    \put(128.5,16){\tiny $1$}
    \put(143,28){\tiny $1$}
    \put(114,3){\tiny $1$}
    \put(142,10){\tiny $0$}
    \put(118,34){\vector(1,1){10}}
    \put(129,21){\vector(0,1){23}}
    \put(141,16){\vector(-1,3){9.5}}
    \put(144,34){\vector(-1,1){10}}
    \put(116,26){\vector(0,-1){18}}
    \put(118,7){\vector(1,1){10}}
    \put(144,16){\vector(0,1){10}}
    \put(132.5,19){\vector(1,1){10}}
    \put(157,23){\small $\oplus$}
    \put(189,46){\tiny $1$}
    \put(175,28){\tiny $1$}
    \put(188.5,16){\tiny $0$}
    \put(203,28){\tiny $0$}
    \put(173,3){\tiny $0$}
    \put(202,10){\tiny $0$}
    \put(178,34){\vector(1,1){10}}
    \put(189,21){\vector(0,1){23}}
    \put(201,16){\vector(-1,3){9.5}}
    \put(204,34){\vector(-1,1){10}}
    \put(176,26){\vector(0,-1){18}}
    \put(178,7){\vector(1,1){10}}
    \put(204,16){\vector(0,1){10}}
    \put(192.5,19){\vector(1,1){10}}
    \put(217,23){\small $\oplus$}    
    \put(249,46){\tiny $1$}
    \put(235,28){\tiny $0$}
    \put(248.5,16){\tiny $1$}
    \put(263,28){\tiny $1$}
    \put(233,3){\tiny $0$}
    \put(262,10){\tiny $1$}
    \put(238,34){\vector(1,1){10}}
    \put(249,21){\vector(0,1){23}}
    \put(261,16){\vector(-1,3){9.5}}
    \put(264,34){\vector(-1,1){10}}
    \put(236,26){\vector(0,-1){18}}
    \put(238,7){\vector(1,1){10}}
    \put(264,16){\vector(0,1){10}}
    \put(252.5,19){\vector(1,1){10}}
  \end{picture}
\caption{All decompositions of a graph.}
\phantomsection\label{figure:decompositions}
\end{figure}
\end{center}

Since the sum of two standard graphs with the same nodes and edges is
standard, a labeled graph has a standard decomposition only if it is
standard. Proposition \ref{pro:extend} shows that the converse is also true.

\begin{dfn}[Maximal standard component]
The \emph{maximal standard component} of a standard graph $G$ is the
unique standard component $H$ for which $\lab H v=1$ if, and only if,
$\lab G v>0$.
\end{dfn}

The standard component is maximal in the sense that it contains all
other standard components. Note that the maximal standard component is
always a standard component unless we are in the degenerate case where
all nodes of $G$ are labeled zero.

\begin{pro}
\phantomsection\label{pro:extend}
Let $\hsym$ be a multiset of standard components of a labeled graph
$G$. Then $\hsym$ can be extended to a standard decomposition of $G$ if,
and only if, $\gsub G {\sum\hsym}$ is standard. In particular, a
standard graph admits a standard decomposition. 
\end{pro}
\begin{proof}
\proofPart{\text{if}} If $G=\sum\hsym$ then we are done, so suppose that
$G\neq\sum\hsym$. Let $C$ be the maximal standard component of
$\gsub G {\sum\hsym}$. Then $\gsub G {\sum(\hsym\cup\set C)}$ is standard. 
The assertion follows from this by induction.

\proofPart{\text{only if}} If $\hsym\p$ is a multiset of standard
components of $G$ that contains $\hsym$, and $\gsub G {\sum\hsym}$ is not
standard, then neither is $\gsub G {\sum\hsym\p}$, so $\hsym\p$ is not a
standard decomposition of $G$.

The last statement of the proposition follows from the first by taking $\hsym=\emptyset$. 
\end{proof}

\begin{cor}
\phantomsection\label{cor:compExtend}
If $H$ is a standard component of a standard graph $G$, then $H$ is an
element of at least one standard decomposition of $G$.
\end{cor}

We conclude this section with a remark on the class of graphs under
consideration.  We stated earlier in this section that loops and
parallel edges are not interesting in the theory of standard graphs.
We now state that
\begin{itemize}
  \item neither are cycles in $G$,
  \item nor nodes with label zero, 
  \item nor edges $(a,b)\in\edges G$ with $\lab G a=\lab G b$.
\end{itemize}
To support these statements, let $H$ be a standard component of $G$
and let $(a,b)$ be an edge of $G$ such that $\lab G a=\lab G b$. Then
$\lab H a=1$ if, and only if, $\lab H b=1$.  So to study standard
decompositions, we might as well suppress all such edges $(a,b)$,
merging the nodes $a$ and $b$ into a single node $ab$. We then
replace  each edge with source either $a$ or $b$ by an edge with
source $ab$ and the same target as before, 
and each edge with target either $a$ or $b$ by an edge with target
$ab$ and the same source as before.    If a standard graph 
has a cycle, then the labels along the cycle are all equal, so
contracting same-label edges removes all cycles.  Nodes with label
zero are not relevant for standard decomposition either, so we can get
rid of those nodes too. 
Combining these
ideas, we get the notion of a \emph{canonical labeled graph}.

\begin{dfn}[Canonical labeled graph]
\phantomsection\label{dfn:canlabeledGraph}
A labeled graph $G$ is \emph{canonical} if
\begin{enumerate}
\item $G$ is standard; 
\item $G$ has no cycles; 
\item all labels are positive; 
\item $\lab G a<\lab G b$ for all edges $(a,b)\in\edges G$.
\end{enumerate}
\end{dfn}

From the above discussion, we have proved: 
\begin{pro}
\phantomsection\label{pro:canonicalizatedGraph}
  For each standard graph $G$, there is a canonical graph
  $G\p$ such that the standard decompositions and standard components of
  $G$ and $G\p$ are related by a bijection. 
\end{pro}


We call the process of
replacing $G$ with $G\p$ \emph{canonicalization}. This notion of
canonicalization is not to be confused with the usual notion of graph
canonicalization, which has to do with isomorphism classes of graphs.

Canonicalizations of graphs are useful to speed-up computer programs.
We will return to the topic of
canonicalization in Section \ref{sec:canonicalization}.


\section{Standard node decompositions}

We now turn to the topic of the computational complexity of the
problem of computing standard decompositions.  We start with a simple
instructive example.

\begin{ex}
\phantomsection\label{ex:bigOutput}
Let $G_n$ be the labeled graph defined by
\[
\nodes {G_n}\defeq\set{y,x_1,\ldots,x_n},\quad\quad
\edges {G_n}\defeq\set{(x_1,y),\ldots,(x_n,y)},
\]
and $\lab {G_n}{x_i}\defeq 1\text{ for }i=1,\ldots,n$, while $\lab
{G_n} y\defeq 2$. There are $2^n$ standard components of $G_n$, 
corresponding to the $n$ independent choices of whether to include or
exclude each $x_i$. The standard decompositions of $G_n$ are pairs of
standard components that include complementary subsets of
$\set{x_1,\ldots,x_n}$. So $G_n$ has $2^{n-1}$ standard decompositions
while having only $n+1$ nodes.
\end{ex}

Consider the computational problem whose input is a labeled graph $G$
and whose output is the set of standard decompositions $\decom G$.
Recall that $\decom G \neq \emptyset$ if, and only if, $G$ is standard
--- however, we will formulate our statements for arbitrary labeled
graphs, thus covering also the case where the output is the empty set.
Example \ref{ex:bigOutput} shows that this computation cannot be done
in time better than exponential in the worst case since just writing
down the output can take exponential time. For problems such as this,
it is standard practice to consider an alternative notion of
complexity, \emph{generating complexity}, in which we consider the
running time as a function of the combined size of the input
\emph{and} the output.

We present an algorithm for standard decomposition of graphs that runs
in polynomial time in the combined size of input and output. This
algorithm is based on the following notion of decomposing a single
node of a standard graph.

\begin{dfn}[Standard node decomposition]
\phantomsection\label{dfn:stdNodeDecomposition}
Let $G$ be a labeled graph and let $v$ be a node of $G$. A multiset of
standard graphs $\hsym$ is a \emph{standard $v$-decomposition} of $G$
if
\begin{enumerate}
\item each $H\in\hsym$ is a standard component of $G$,
\item $\lab H v=1$ for all $H\in\hsym$,
\item $\gsub G {\sum\hsym}$ is standard,
\item $\card H=\lab G v$.
\end{enumerate}
We denote the set of standard $v$-decompositions of $G$ by $\decomn G v$.
\end{dfn}

Consider a standard graph $G$ with a standard decomposition $\hsym$
and a node $v$ of $G$. The submultiset of $\hsym$ whose elements give
$v$ a label of 1 forms a standard $v$-decomposition of $G$. Another
way of characterizing a standard $v$-decomposition is that it is a
minimal multiset $\hsym$ of standard components of $G$ such that
$\gsub G \sum\hsym$ gives $v$ the label $0$ and such that $\hsym$ can
be extended to a standard decomposition of $G$.

Every standard graph has at least one standard decomposition, so
Proposition \ref{pro:nodeDecom} implies that if we can generate
standard $v$-decompositions in polynomial time, then we can also
generate standard decompositions in polynomial time.

\begin{pro}
\phantomsection\label{pro:nodeDecom}
Let $v$ be a node of a labeled graph $G$. Then
\[
\decom G=
\setBuilderBarRight{\hsym \cup \hsym\p}{\hsym\in\decomn G v, \hsym\p \in \decom{\gsub G \sum\hsym}},
\]
where no decomposition appears twice on the right hand side.
\end{pro}

\begin{proof}
\proofPart{\subseteq} Let $D\in\decom G$ and let $\hsym$ be the
submultiset of $D$ whose elements give $v$ the label 1. Then
$\hsym\in\decomn G v$. It only remains to prove that
$D\setminus\hsym\in\decom{\gsub G \sum\hsym}$, which follows from Lemma
\ref{lem:partDecom} below.

\proofPart{\supseteq} Let $\hsym\in\decomn G v$ and let
$\hsym\p\in\decom{\gsub G \sum\hsym}$. Then $\hsym\p\cup\hsym$ is a standard
decomposition of $G$ by Lemma \ref{lem:partDecom}.

\proofPart{\text{no duplicates}} Let $\hsym,\hsym\pp\in\decomn G v$
such that $\hsym\neq\hsym\pp$. Let $A\in\decom{\gsub G \sum\hsym}$ and
$A\pp\in\decom{\gsub G \sum\hsym\pp}$. Then $\hsym\cup A\neq\hsym\pp\cup A\pp$ since
$\hsym\neq\hsym\pp$ and $A\cup A\pp$ is disjoint from $\hsym\cup\hsym\pp$, 
as the elements of $A\cup A\pp$ give $v$ the label zero, while the
elements of $\hsym\cup\hsym\pp$ give $v$ the label 1.
\end{proof}

\begin{lmm}
\phantomsection\label{lem:partDecom}
Let $G$ be a labeled graph. Let $A$ be a multiset of standard 0-1
subgraphs of $G$ and let $B$ be a submultiset of $A$. Then $A$ is a
standard decomposition of $G$ if, and only if, $A\setminus B$ is a
standard decomposition of $\gsub G \sum B$.
\end{lmm}

\begin{proof}
 \proofPart{\text{if}} Assume that $A\setminus B$ is a standard
 decomposition of $\gsub G \sum B$. Then $\gsub G \sum
 B=\sum(A\setminus B)$ so $G=\sum A$. It only remains to prove that
 each $a\in A$ is a standard component of $G$. To prove that, we need
 to show that $\gsub G a$ is standard. We already know that $a$ is a
 standard component of $\gsub G \sum B$, which implies that $\gsub G
 {\gsub {\sum B} a}$ is standard.  Then $\gsub G a=\gadd {(\gsub G
   {\gsub {\sum B} a})} \sum B$ is standard, as it is a a sum of
 standard graphs with identical node sets.

 \proofPart{\text{only if}} Assume that $A$ is a standard decomposition
 of $G$. Then $G=\sum A$ so $\gsub G \sum B=\sum(A\setminus B)$. It
 only remains to prove that each $a\in A\setminus B$ is a standard
 component of $\gsub G \sum B$. To prove that we need to show that
 $\gsub G {\gsub {\sum B} a}$ is standard. We already know that $\gsub
 G \sum A$ has all labels zero, so it is standard.  Then $\gsub G
 {\gsub {\sum B} a}=\gadd {(\gsub G \sum A)} \sum(A\setminus(B\cup\set
 a)))$, so $\gsub G {\gsub {\sum B} a}$ is standard, as it is a sum of
 standard graphs with identical node sets.
\end{proof}


\section{Generating standard node decompositions}
\phantomsection\label{sec:nodeDecom}

Proposition \ref{pro:nodeDecom} reduces the problem of generating
$\decom G$ in polynomial time to the problem of generating the
standard node decomposition $\decomn G v$ in polynomial time for some
freely chosen node $v$ of $G$. In this section we investigate this
problem. Our solution is based on choosing the right node $v$ to
decompose.

Consider the set of all standard components of $G$ that give $v$ the
label 1.  We impose an ordering, $H_1,\ldots,H_k$, on the elements of
that set.  This ordering can be chosen arbitrarily, but is fixed once
and for all.  Now let $F$ be any labeled subgraph of $G$.
For each such $F$ and each $i = 1, \ldots, k$, we define
\[
\tau(F,i)\defeq\setBuilder{\hsym\subseteq\set{H_1,\ldots,H_i}}{\hsym \in \decomn
  F v}
\]
Note that $\decomn F v$ and $\tau(F,i)$ are both sets of multisets,
thus the condition ${\hsym \in \decomn
  F v}$ for a multiset $\hsym$ makes sense; 
If we can compute $\tau(F,i)$ in general then we can also compute
$\decomn G v$ since $\tau(G,k)=\decomn G v$. In order to compute
$\tau(F,i)$, consider the recursive formula
\begin{equation}
\phantomsection\label{eqn:nodeDecomRecursive}
\tau(F,i)=
\begin{cases}
\set{\emptyset} \text{ if all labels of $F$ are zero, else}\\
\emptyset \text{ if $F$ is not standard or $i=0$, else}\\
\tau(F,i-1)\cup
  \setBuilder{\hsym\cup\set{H_i}}{\hsym\in\tau(\gsub F H_i,i)}.
\end{cases}
\end{equation}
This way of writing $\tau$ immediately suggests an algorithm based on
recursively evaluating the expression.  It is a problem with this
approach that this algorithm can spend a large amount of computational
steps to determine that $\tau(F,i)$ is empty. This is an obstacle to
proving that this algorithm generates its output in polynomial time.

We say that a pair $(F,i)$ is \emph{relevant} if
$\tau(F,i)\neq\emptyset$, and \emph{irrelevant} otherwise.\footnote{In particular,
we see that it suffices to consider graphs such that $0 \leq \lab F w \leq \lab G w$ for all nodes $w$, 
since $(F,i)$ is irrelevant otherwise.}
For making the algorithm generate its output in polynomial time, we
need a criterion for detecting irrelevant pairs.  We can use such a
criterion to quickly eliminate irrelevant pairs in the algorithm.

\begin{pro}
\phantomsection\label{pro:relevantCriterion}
Let $v$ be a node of minimal positive label in a labeled graph
$G$. Let $\hsym$ be a multiset of standard components of $G$ that give
$v$ the label 1. Let $H$ be the maximal standard component of $G$.
Assume that $\gsub G {\sum\hsym}$ is standard. Let $\hsym\p$ be the union of
$\hsym$ and the multiset containing $\lab G v-\card\hsym$ copies of
$H$. Then $\hsym\p$ is a standard $v$-decomposition of $G$.
\end{pro}

\begin{proof}
Upon applying the proof of Proposition \ref{pro:extend} to $\hsym$, 
we obtain a standard decomposition $\hsym\pp\supseteq\hsym$ of $G$. 
Since the label of $v$ is minimal among all positive labels appearing in $G$, 
the first $\lab G v-\card\hsym$ rounds of the
inductive construction in that proof will use the same maximal
standard component $H$. After that the label of $v$ has become zero,
so the maximal standard components used in later rounds of the
construction will give $v$ the label zero. So the subset of
$\hsym\pp$ that gives $v$ the label 1 is precisely $\hsym\p$, which
implies that $\hsym\p$ is a standard $v$-decomposition of $G$.
\end{proof}

Through choosing wisely the node $v$ and the order of the standard
components $H_1,\ldots,H_k$, Proposition \ref{pro:relevantCriterion>0}
gives an if-and-only-if criterion for detecting irrelevant pairs.

\begin{pro}
\phantomsection\label{pro:relevantCriterion>0}
Given a labeled graph $G$, choose $v$ to be a node of minimal positive
label, and choose an order on the standard components $H_1,\ldots,H_k$
giving $v$ the label 1 such that $H_1$ is the maximal standard
component of $G$.  Let $\hsym$ be a multiset whose elements are chosen
among the standard components $H_i$ of $G$, and let $F\defeq\gsub G
{\sum\hsym}$.
 Then a pair $(F,i)$ with $1\leq i\leq k$ is
relevant if, and only if, $F$ is standard.
\end{pro}

\begin{proof} 
\proofPart{\text{if}} Assume that $F$ is standard. 
By Proposition  \ref{pro:relevantCriterion}, $G=\sum \hsym \oplus \sum
\hsym_1 \oplus \sum \hsym_2$, where $\hsym_1$ is a multiset containing
copies of $H_1$ and $\hsym_2$ is a multiset containing standard
components of $G$ with label $0$ on $v$. Thus $F= \sum
\hsym_1 \oplus \sum \hsym_2$, and $\tau(F,i)$ is not empty. 

\proofPart{\text{only if}} This part is obvious. 
\end{proof}


\section{Generating standard decompositions in polynomial time}

Based on the previous two sections, we can now present an algorithm
for generating standard decompositions and prove that it runs in
polynomial time.

\begin{thm}
\phantomsection\label{thm:algorithm}
The algorithm in Figure \ref{figure:alg} generates the standard
decompositions of a labeled graph in polynomial time.
\end{thm}

\begin{figure}
\begin{algorithmic}[1]
  \Function{standardDecompositions}{$G$}
    \If{all labels of all nodes of $G$ are zero}
      \State\Return $\set{\emptyset}$
    \Else
      \State choose a node $v \in \nodes G$ of minimal positive label
      \State $D\gets \textsc{StandardNodeDecompositions}(G,v)$
      \State \Return
      $\setBuilderBarLeft{\hsym \cup \hsym\p}{\hsym\in D,
      \hsym\p \in \textsc{standardDecompositions}(\gsub G {\sum \hsym})}$ \phantomsection\label{line:combine}
    \EndIf
  \EndFunction
  \Function{standardNodeDecompositions}{$G$, $v$}
    \State $H_1 \gets$ the maximal standard component of $G$
    \State $H_2, \ldots, H_k \gets$ all other standard components of $G$ that give $v$ the label 1
    \State $S \gets \set{H_1, \ldots, H_k}$
    \State \Return $\textsc{tau}(G,k,S)$
  \EndFunction
\Function{tau}{$F$, $i$, $S$}
    \If {all labels of $F$ are zero}
      \State \Return $\set \emptyset$
    \Else
      \If {$F$ is not a standard graph} \phantomsection\label{line:disposeOfIrrelevantPairs}
        \State \Return $\emptyset$ 
      \Else
        \State \Return $\textsc{tau}(F,i-1,S) \cup 
        \setBuilder{\hsym\cup\set{H_i}}{\hsym\in\textsc{tau}(\gsub F H_i,i,S)}$
      \EndIf
    \EndIf
  \EndFunction
\end{algorithmic}
\caption{An algorithm for standard decomposition.}
\phantomsection\label{figure:alg}
\end{figure}

The pseudo code for {\tt standardDecompositions} implements the
recursive formula from Proposition \ref{pro:nodeDecom}. The pseudo code for {\tt
  standardNodeDecompositions} implements the recursion from
Section \ref{sec:nodeDecom} where the function {\tt Tau} is $\tau$
from that section. Line \ref{line:disposeOfIrrelevantPairs} eliminates
pairs that are irrelevant according to Proposition
\ref{pro:relevantCriterion>0}.

In reading the pseudo code for {\tt Tau}, note that the first return
is of the value $\set\emptyset$ while the second is of the value
$\emptyset$. Here $\set\emptyset$ is a set containing one
decomposition while $\emptyset$ is a set containing nothing.

\begin{proof}[Proof of Theorem \ref{thm:algorithm} 
and thus also of Theorem \ref{thm:generatingComplexity}]
Recall that \emph{generating output in polynomial time} means that the
algorithm runs in polynomial time in the combined size of input and
output --- this is the meaning of the word ``generate'' in this
context.

The size of the input and output depend on the representation used. We
specify a graph as a list of nodes with labels and a list of edges. We
specify the set of decompositions as a list of standard components
followed by a list of sets that specify a decomposition by referring
back to the list of components. Each standard component is specified
by a bit per node indicating whether that node is an element of the
standard component.

We assume a model where all labels and indices take up one word of
space, rather than the logarithmic number of bits actually necessary
to hold these numbers. The only arithmetic operations we perform is
subtractions $a-b$ where $a>b$ so this assumption does not weaken the
theorem.

\proofPart{\text{{\tt standardNodeDecompositions} is correct}} Suppose
that we call the function {\tt standardNodeDecompositions} on the pair
$(G,v)$. We know that $v$ is a node of minimal positive label in $G$
since {\tt standardDecompositions} always makes calls to {\tt
  standardNodeDecompositions} with such a $v$. Also observe that the
sequence $H_1,\ldots,H_n$ are ordered to satisfy the precondition of
Proposition \ref{line:disposeOfIrrelevantPairs}. We then see that {\tt
  standardNodeDecompositions} computes the correct value $\decomn G v$
since it directly implements the recursive formula from equation
\ref{eqn:nodeDecomRecursive} along with the criterion for irrelevant
pairs from Proposition \ref{pro:relevantCriterion>0}.

\proofPart{\text{{\tt standardNodeDecompositions} is polynomial}} Let
$G$ have $n$ nodes and $e$ edges. We do not give pseudo code for
generating $H_1,\ldots,H_k$, but it is not difficult to do this in
time $O(k(n+e))$ using backtracking. We first need to prove that
$k(n+e)$ is polynomial in the size of the output.

Let $l$ be the label of $v$ in $G$. Every $H_i$ is an element of at
least one standard decomposition of $G$ by Corollary
\ref{cor:compExtend}, and each $v$-decomposition has exactly $l$
elements, so $k\leq ld$ where $d$ is the number of standard
$v$-decompositions of $G$. So computing $H_1,\ldots,H_k$ can be done
in time $O(ld(n+e))$. The size of the input is $\Theta(n+e)$ and the
size of the output is $\Theta(ld+kn)$ since it takes $l$ elements of
$S$ to specify each of the $d$ decompositions and for each irreducible
decomposition we need one bit per node to specify whether it is in the
graph or not. Clearly $ld(n+e)=\Omega(ldn^2)$ is bounded above by a
polynomial in $ld+kn$, so the time to compute $S$ is polynomial.

It remains to prove that ${\tt Tau}$ takes polynomial time. Each
individual call to {\tt Tau}, not counting recursive subcalls, can be
done in time $O(n+e)$. We need an upper bound for the number of
recursive calls.

Consider a tree $T$ where each recursive call to {\tt Tau} is a node
labeled by the parameters $(F,i)$ and where there is an edge from the
caller to the callee. The relevant leaves of $T$ give rise to one
distinct node decomposition per leaf so $d$, the number of
$v$-decompositions of $G$, is also the number of relevant leaves of
$T$. Let $r$ be the number of irrelevant leaves of $T$ --- these do
not give rise to a $v$-decomposition. Since $T$ is a binary tree we
see that there are $r+d-1$ internal nodes in $T$. We need an upper
bound for $r$.

Since Proposition \ref{pro:relevantCriterion>0} is an if-and-only-if
criterion for irrelevant pairs, we see that the sub-tree rooted at any
internal node contains a relevant pair. This implies that the sibling
of an irrelevant leaf $A$ is a root of a sub-tree that contains some
relevant leaf $B$. Let $f$ be the mapping $A\mapsto B$. If $f(A)=B$
then the parent of $A$ is on the path from the root of $T$ to $B$. All
the relevant leaves are at depth $k$ or less, so $f$ can map at most
$k$ irrelevant leaves to each relevant leaf. This implies that $r\leq
dk$.

We have seen that there are $d$ relevant leaves, at most $dk$
irrelevant leaves and therefore also at most $d+dk$ internal nodes in
$T$, which is a total of at most $2d+2dk$ nodes. So the time taken by
all recursive calls to {\tt Tau} is $O(dk(n+e))$. Recall that the
input size is $\Theta(n+e)$ and the output size is
$\Theta(ld+kn)$. Clearly $dk(n+e)$ is dominated by a polynomial in
$(n+e)+(ld+kn)$. This proves that {\tt standardNodeDecompositions} generates
$\decomn G v$ in polynomial time.

\proofPart{\text{{\tt standardDecompositions} is correct}} We have
already done the correctness proof since {\tt standardDecompositions}
directly implements the recursive formula for $\decom G$ from
Proposition \ref{pro:nodeDecom}.

\proofPart{\text{{\tt standardDecompositions} is polynomial}} We have
already seen that each call to {\tt standardNodeDecompositions}
generates its own output in polynomial time. Consider as before a tree
$T$ where each recursive call to {\tt standardDecompositions} is a
node with an edge from the caller to the callee. Let $q$ be the number
of leaves of $T$. Every leaf contributes at least one distinct
decomposition to the output, so $q$ is a lower bound on the number of
decompositions of $G$. The multiset of node decompositions computed by
all the calls to {\tt standardNodeDecompositions} is in bijection with
the edges of $T$. All trees have more nodes than edges and more leaves
than internal nodes so the combined time to compute all the node
decompositions is dominated by a polynomial in $q(n+e)$ where $n+e$ is
the input size for the original input which is an upper bound on the
size of any graph produced during the computation.

Line \ref{line:combine} could a priori seem to require too much time
by going through all the elements of $D$. However, we can charge this
work to each of the children of that node that are produced in this
way which clears up the problem. As trees have more leaves than
internal nodes the total number of nodes of $T$ is less than
$2q$. This proves that the total time to compute $\decom G$ is bounded
by a polynomial in $w(n+e)$ where $w$ is the number of decompositions
and $\Theta(n+e)$ is the size of the input.
\end{proof}

We can extract some bounds on the number of node decompositions from
the arguments just given.

\begin{pro}
\phantomsection\label{pro:decomBounds}
Let $v$ be a node of a standard graph $G$. Let $l\defeq\lab G v$. If
$G$ has $k$ standard components that give $v$ label 1, then there
are between $\frac{k}{l}$ and $\binom {k+l-1} l$ standard $v$-decompositions
of $G$. If $v$ is a node of minimal positive label in $G$, then there
are at least $k$ standard $v$-decompositions of $G$.
\end{pro}

\begin{proof}
Every $v$-decomposition of $G$ has exactly $l$ elements, and the
elements of each such multiset are chosen among the $k$ standard components that give $v$ the 
label 1, so there cannot be more than $\binom {k+l-1} l$ standard
$v$-decompositions.

Every one of the $k$ standard components giving $v$ label 1 can be extended to a standard
decomposition of $G$ by Corollary \ref{cor:compExtend} and therefore
also to a standard $v$-decomposition. We get the minimal number of
standard $v$-decompositions when each of these extensions are
unique. As each standard $v$-decomposition has $l$ elements, that
implies the existence of at least $\frac{k}{l}$ standard $v$-decompositions.

If $v$ is a label of minimal positive label, then each standard
component $H$ that gives $v$ the label 1 can be extended to a
$v$-decomposition using only the maximal standard component by Proposition
\ref{pro:relevantCriterion}. So there are at least $k$ standard
$v$-decompositions in this case.
\end{proof}

Here are examples in which the bounds from the proposition are sharp. 

\begin{center}
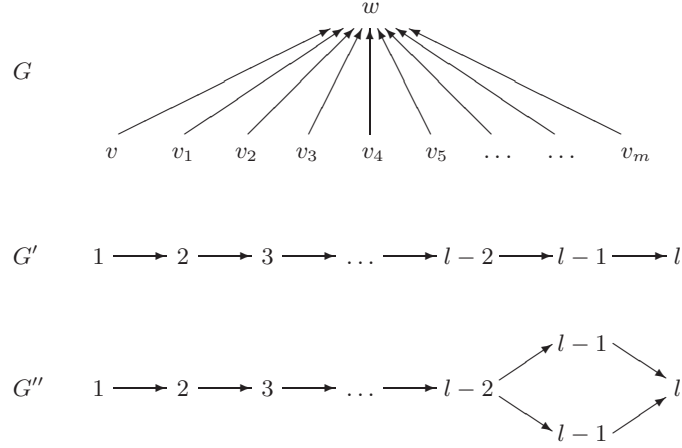
\begin{figure}[ht]
  \begin{picture}(255,170)
    \put(0,140){\small $G$}
    \put(40,119){\vector(2,1){80}}
    \put(65,119){\vector(3,2){60}}
    \put(89,119){\vector(1,1){40}}
    \put(112,119){\vector(1,2){20}}
    \put(135,119){\vector(0,1){40}}
    \put(158,119){\vector(-1,2){20}}
    \put(181,119){\vector(-1,1){40}}
    \put(205,119){\vector(-3,2){60}}
    \put(230,119){\vector(-2,1){80}}
    \put(132,165){\small $w$}
    \put(35,110){\small $v$}
    \put(60,110){\small $v_1$}
    \put(84,110){\small $v_2$}
    \put(107,110){\small $v_3$}
    \put(132,110){\small $v_4$}
    \put(156,110){\small $v_5$}
    \put(178,110){\small $\ldots$}
    \put(202,110){\small $\ldots$}
    \put(230,110){\small $v_m$}
    \put(0,70){\small $G\p$}
    \put(30,70){\small $1$}
    \put(38,73){\vector(1,0){20}}
    \put(62,70){\small $2$}
    \put(70,73){\vector(1,0){20}}
    \put(94,70){\small $3$}
    \put(102,73){\vector(1,0){20}}
    \put(126,70){\small $\ldots$}
    \put(140,73){\vector(1,0){20}}
    \put(163,70){\small $l-2$}
    \put(184,73){\vector(1,0){20}}
    \put(206,70){\small $l-1$}
    \put(227,73){\vector(1,0){20}}
    \put(250,70){\small $l$}
    \put(0,20){\small $G\pp$}
    \put(30,20){\small $1$}
    \put(38,23){\vector(1,0){20}}
    \put(62,20){\small $2$}
    \put(70,23){\vector(1,0){20}}
    \put(94,20){\small $3$}
    \put(102,23){\vector(1,0){20}}
    \put(126,20){\small $\ldots$}
    \put(140,23){\vector(1,0){20}}
    \put(163,20){\small $l-2$}
    \put(184,26){\vector(3,2){20}}
    \put(184,20){\vector(3,-2){20}}
    \put(206,37){\small $l-1$}
    \put(206,3){\small $l-1$}
    \put(227,40){\vector(3,-2){20}}
    \put(227,6){\vector(3,2){20}}
    \put(250,20){\small $l$}
  \end{picture}
\caption{Three graphs leading to sharp bounds in Proposition \ref{pro:decomBounds}}
\phantomsection\label{figure:sharpBounds}
\end{figure}
\end{center}

\begin{ex}
  Consider the graph $G$ from Figure \ref{figure:sharpBounds}, 
  whose labels we will presently specify, and the graphs $G\p$ and $G\pp$ from the same figure, 
  whose labels are specified in the picture. 
  \begin{itemize}
    \item Choose the labels such that $\lab G {v_i} \geq \lab G v$ for all $i$ 
    and $\lab G w \geq \lab G v + \sum_i \lab G {v_i}$. 
    Then for each multiset of standard graphs $\hsym$ satisfying conditions (1), (2) and (4) 
    from Definition \ref{dfn:stdNodeDecomposition}, condition (3) is automatically satisfied. 
    The number $k$ from the proposition depends on the choice of the labels, 
    but in any case, $G$ has $\binom {k+l-1} l$ standard $v$-decompositions. 
    The upper bound for the number of standard $v$-decompositions is therefore sharp. 
 \item Define $\lab G v \defeq 1$, $\lab G {v_i} \defeq 1$ for all $i$ and $\lab G w \defeq 2$. 
     Then the standard components that give $v$ the label 1 correspond to 
     the power set of $\set{v_1, \ldots, v_m}$, whose cardinality is $2^m$. 
     All standard components that give $v$ the label 1 lead to standard $v$-decompositions of $G$. 
     The lower bound for the number of standard $v$-decompositions is therefore sharp. 
    \item The graph $G\p$ provides another example of sharpness of the lower bound, this time with $l>1$.
    We define $v$ as the node of label $l$. 
    As in the proposition, we denote by $k$ the number of standard components of $G\p$ that give $v$ the label 1. 
    Since $v$ is labeled 1 in every standard component, $k$ is just the number of components of $G\p$. 
    Likewise, a standard $v$-decomposition of $G\p$ is just a standard decomposition of $G\p$. 
    Obviously $k = l$, and there exists precisely one standard $v$-decomposition. 
    \item Also in the graph $G\pp$, we define $v$ as the node of label $l$. 
    This graph has the property that the lower bound is sharp
    while, unlike in the previous example, there exists more than one standard $v$-decomposition. 
    Note that the fraction $\frac{k}{l} = \frac{2l -1}{l}$ is not an integer, but $\lceil \frac{k}{l} \rceil = 2$. 
  \end{itemize}
\end{ex}

We leave the question open whether there exist $k$ and $l$ as in the proposition such that $\frac{k}{l} > 2$ 
and there exists a graph $G$ such that the lower bound from the proposition is sharp. 


\section{From standard sets to standard graphs}

In the remaining three sections, we investigate the relation
between standard decomposition of labeled graphs and another
combinatorial problem called \emph{Connect Four
  decomposition}. 
In the end we show that the two problems
are equivalent.


A \emph{standard set}, or \emph{staircase}, is a subset $\Delta \subseteq \N^d$ whose
complement $C \defeq \N^d\setminus \Delta$ satisfies $C+\N^d = C$. 
We are only going to consider standard sets of finite cardinalities. 
Standard sets in $\N$ are just intervals starting at $0$; 
in $\N^2$, they can be identified with partitions, or with Young diagrams\footnote{in the French notation}; 
in $\N^3$, they are also known as \emph{plane partitions}; 
in $\N^d$ for $d>3$, they are also known as \emph{solid partitions}. 
Standard sets in $\N^d$ canonically correspond to monomial ideals in the polynomial ring $k[x_1, \ldots, x_d]$. 
See Figure \ref{figure:standardSets} for examples in dimensions 1, 2, and 3. 

\begin{center}
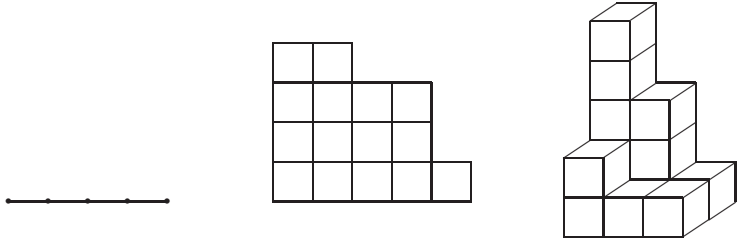
\begin{figure}[ht]
  \begin{picture}(275,130)
    \put(0,33.6){\line(1,0){60}}
    \multiput(0,33.6)(15,0){5}{\circle*{2}}
    \multiput(100,33.6)(15,0){3}{\line(0,1){60}}
    \multiput(145,33.6)(15,0){2}{\line(0,1){45}}
    \put(175,33.6){\line(0,1){15}}
    \multiput(100,33.6)(0,15){2}{\line(1,0){75}}
    \multiput(100,63.6)(0,15){2}{\line(1,0){60}}
    \put(100,93.6){\line(1,0){30}}
    \put(210,20){\line(1,0){45}}
    \put(210,20){\line(0,1){30}}
    \put(210,35){\line(1,0){45}}
    \put(210,50){\line(1,0){15}}
    \put(225,20){\line(0,1){30}}
    \put(240,20){\line(0,1){15}}
    \put(255,20){\line(0,1){15}}
    \put(210,50){\line(3,2){10}}
    \put(225,50){\line(3,2){10}}
    \put(225,35){\line(3,2){10}}
    \put(240,35){\line(3,2){20}}
    \put(255,35){\line(3,2){20}}
    \put(255,20){\line(3,2){20}}
    \put(220,56.8){\line(0,1){45}}
    \put(235,41.8){\line(0,1){60}}
    \put(250,41.8){\line(0,1){30}}
    \put(250,41.8){\line(1,0){15}}
    \put(265,26.8){\line(0,1){15}}
    \put(220,71.8){\line(1,0){30}}
    \put(220,86.8){\line(1,0){15}}
    \put(220,101.8){\line(1,0){15}}
    \put(220,56.8){\line(1,0){30}}
    \put(235,41.8){\line(1,0){15}}
    \put(220,101.8){\line(3,2){10}}
    \put(235,71.8){\line(3,2){10}}
    \put(235,86.8){\line(3,2){10}}
    \put(235,101.8){\line(3,2){10}}
    \put(250,71.8){\line(3,2){10}}
    \put(250,56.8){\line(3,2){10}}
    \put(260,78.6){\line(-1,0){15}}
    \put(245,78.6){\line(0,1){30}}
    \put(230,108.6){\line(1,0){15}}
    \put(260,78.6){\line(0,-1){30}}
    \put(260,48.6){\line(1,0){15}}
    \put(275,33.6){\line(0,1){15}}
  \end{picture}
\caption{Standard sets in dimensions 1, 2 and 3}
\phantomsection\label{figure:standardSets}
\end{figure}
\end{center}

Consider the projection to the first $d-1$ components, 
$q^d : \N^d \to\N^{d-1} : \beta \mapsto (\beta_1,\ldots,\beta_{d-1})$ 
and its complementary projection, 
$q_d : \N^d \to\N : (\beta_1,\ldots,\beta_d) \mapsto \beta_d$
For each standard set $\Delta$, we have the equality 
\[
  \Delta = \setBuilderBarLeft{\beta \in \N^d}{q_d(\beta) < \card{(q^d)^{-1}(q^d(\beta)) \cap \Delta}}. 
\]
The integer $\card{(q^d)^{-1}(q^d(\beta)) \cap \Delta}$ appearing on the right-hand side is the 
cardinality of the fiber of the projection $q^d : \Delta \to \N^{d-1}$ over the point $\gamma \defeq q^d(\beta)$. 
We call that quantity the \emph{height} of $\Delta$ over $\gamma$. 
The equation displayed above implies that the datum of standard set $\Delta$ is equivalent to 
the datum of the projection $\Delta\p \defeq q^d(\Delta)$, which is a standard set in $\N^{d-1}$, 
and the datum of the heights over all $\gamma \in \Delta\p$. 
The heights satisfy a compatibility condition: 
Upon denoting by $h_\gamma$ the height over $\gamma \in \Delta\p$, 
we see that $h_{\gamma+e_i} \leq h_\gamma$ for all standard basis elements $e_i \in \N^{d-1}$ 
and all $\gamma \in \Delta\p$ such that also $\gamma+e_i \in \Delta\p$. 
These observations motivate the following definition: 

\begin{dfn}[Standard graph of a standard set]
\phantomsection\label{dfn:StdSetToPartition}
  Let $\Delta \subseteq \N^d$ be a finite standard set. 
  We define the \emph{standard graph of $\Delta$}, denoted by $G(\Delta)$, by setting 
  \[
    \begin{split}
      \nodes {G(\Delta)} & \defeq q^d(\Delta) , \\
      \edges {G(\Delta)} & \defeq \setBuilderBarLeft {(\gamma\p,\gamma)}{\gamma\p = \gamma + e_i \text{ for some }i} \\
      \lab {G(\Delta)} \gamma & \defeq \card{(q^d)^{-1}(\gamma) \cap \Delta} . 
    \end{split}
  \]
\end{dfn}

The discussion leading to the definition proves that $G(\Delta)$ is indeed a standard graph. 
The transition from a standard set to its standard graph is illustrated in the first two 
pictures in Figure \ref{figure:standardSetsAndGraphs}. 

Addition of standard graphs has a counterpart on standard sets, called \emph{C4 addition}. 

\begin{dfn}[C4 sum]
\phantomsection\label{dfn:C4Sum}
Let $\Delta_1$ and $\Delta_2$ be two finite standard sets in $\N^{d}$. 
We define the \emph{Connect Four sum}, or \emph{C4 sum} of $\Delta_1$ and $\Delta_2$ by
\[
  \Delta_1 + \Delta_2 \defeq \left\lbrace
  \begin{array}{c|c}
    \beta\in\N^d & 
    q_d(\beta) < \card{ (q^d)^{-1}\bigl(q^d(\beta)\bigr)\cap\Delta_1} \\
    & + \card{ (q^d)^{-1}\bigl(q^d(\beta)\bigr)\cap\Delta_2}
  \end{array}
  \right\rbrace .
\]
\end{dfn}

So for determining the C4 sum of $\Delta_1$ and $\Delta_2$, 
we define $\Delta\p$ to be the union of $q^d(\Delta_1)$ and $q^d(\Delta_2)$ and, 
for all $\gamma \in \Delta\p$, $h_\gamma$ to be 
the sum of the heights over $\gamma$ of $\Delta_1$ and $\Delta_2$.\footnote{We 
say that the height of $\Delta_i$ over $\gamma$ is zero if $\gamma \notin q^d(\Delta_i)$.}
Then $\Delta$ is characterized by its projection $\Delta\p$ and the heights $h_\gamma$. 

Here is a more graphic way of thinking about the C4 sum: 
Place $\Delta_1$ and $\Delta_2$ somewhere on the $d$-axis in $\N^d$ such that they do not intersect, 
subsequently drop the cubes along the $d$-axis, 
until they get stacked above each other on the $1,2,\ldots,(d-1)$-hyperplane. 
The result is the standard set $\Delta_1 + \Delta_2$. 
Figure \ref{figure:C4sums} illustrates that process in two examples. 
The figure also explains the analogy to the eponymous game \emph{Connect Four}. 

\begin{center}
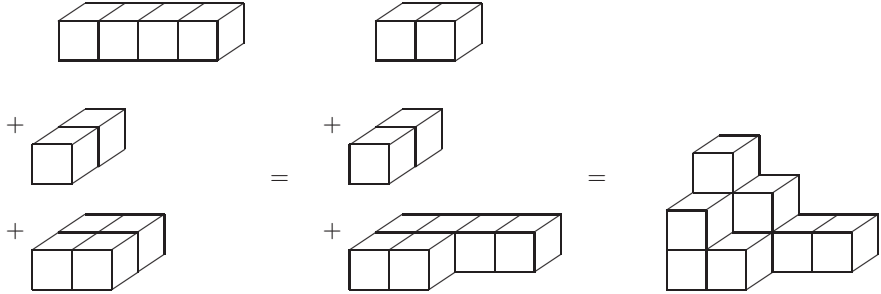
\begin{figure}[ht]
  \begin{picture}(330,140)
    \put(40,20){\line(3,2){20}}
    \put(40,35){\line(3,2){20}}
    \put(25,35){\line(3,2){20}}
    \put(10,35){\line(3,2){20}}
    \put(40,20){\line(0,1){15}}
    \put(25,20){\line(0,1){15}}
    \put(10,20){\line(0,1){15}}
    \put(10,20){\line(1,0){30}}
    \put(10,35){\line(1,0){30}}
    \put(50,26.7){\line(0,1){15}}
    \put(60,33.4){\line(0,1){15}}
    \put(20,41.7){\line(1,0){30}}
    \put(30,48.4){\line(1,0){30}}
    \put(0,40){\small $+$}
    \put(10,60){\line(1,0){15}}
    \put(10,75){\line(1,0){15}}
    \put(30,88.4){\line(1,0){15}}
    \put(10,60){\line(0,1){15}}
    \put(25,60){\line(0,1){15}}
    \put(25,60){\line(3,2){20}}
    \put(25,75){\line(3,2){20}}
    \put(10,75){\line(3,2){20}}
    \put(45,73.4){\line(0,1){15}}
    \put(35,66.7){\line(0,1){15}}
    \put(20,81.7){\line(1,0){15}}
    \put(0,80){\small $+$}
    \put(30,128.4){\line(1,0){60}}
    \put(20,121.7){\line(1,0){60}}
    \put(20,106.7){\line(1,0){60}}
    \multiput(20,106.7)(15,0){5}{\line(0,1){15}}
    \multiput(20,121.7)(15,0){5}{\line(3,2){10}}
    \put(80,106.7){\line(3,2){10}}
    \put(90,113.4){\line(0,1){15}}
    \put(100,60){\small $=$}
    \put(160,20){\line(3,2){10}}
    \put(160,35){\line(3,2){20}}
    \put(145,35){\line(3,2){20}}
    \put(130,35){\line(3,2){20}}
    \put(160,20){\line(0,1){15}}
    \put(145,20){\line(0,1){15}}
    \put(130,20){\line(0,1){15}}
    \put(130,20){\line(1,0){30}}
    \put(130,35){\line(1,0){30}}
    \put(170,26.7){\line(0,1){15}}
    \put(170,26.7){\line(1,0){30}}
    \put(185,26.7){\line(0,1){15}}
    \put(185,41.7){\line(3,2){10}}
    \put(200,41.7){\line(3,2){10}}
    \put(200,26.7){\line(3,2){10}}
    \put(200,26.7){\line(0,1){15}}
    \put(210,33.4){\line(0,1){15}}
    \put(140,41.7){\line(1,0){60}}
    \put(150,48.4){\line(1,0){60}}
    \put(120,40){\small $+$}
    \put(130,60){\line(1,0){15}}
    \put(130,75){\line(1,0){15}}
    \put(150,88.4){\line(1,0){15}}
    \put(130,60){\line(0,1){15}}
    \put(145,60){\line(0,1){15}}
    \put(145,60){\line(3,2){20}}
    \put(145,75){\line(3,2){20}}
    \put(130,75){\line(3,2){20}}
    \put(165,73.4){\line(0,1){15}}
    \put(155,66.7){\line(0,1){15}}
    \put(140,81.7){\line(1,0){15}}
    \put(120,80){\small $+$}
    \put(150,128.4){\line(1,0){30}}
    \put(140,121.7){\line(1,0){30}}
    \put(140,106.7){\line(1,0){30}}
    \multiput(140,106.7)(15,0){3}{\line(0,1){15}}
    \multiput(140,121.7)(15,0){3}{\line(3,2){10}}
    \put(170,106.7){\line(3,2){10}}
    \put(180,113.4){\line(0,1){15}}
    \put(220,60){\small $=$}
    \put(280,20){\line(3,2){10}}
    \put(280,35){\line(3,2){20}}
    \put(265,35){\line(3,2){10}}
    \put(280,20){\line(0,1){15}}
    \put(265,20){\line(0,1){30}}
    \put(250,20){\line(0,1){30}}
    \put(250,20){\line(1,0){30}}
    \put(250,35){\line(1,0){30}}
    \put(290,26.7){\line(0,1){30}}
    \put(290,26.7){\line(1,0){30}}
    \put(305,26.7){\line(0,1){15}}
    \put(305,41.7){\line(3,2){10}}
    \put(320,41.7){\line(3,2){10}}
    \put(320,26.7){\line(3,2){10}}
    \put(320,26.7){\line(0,1){15}}
    \put(330,33.4){\line(0,1){15}}
    \put(275,41.7){\line(1,0){45}}
    \put(300,48.4){\line(1,0){30}}
    \put(275,41.7){\line(0,1){30}}
    \put(300,48.4){\line(0,1){15}}
    \put(260,56.7){\line(1,0){30}}
    \put(290,56.7){\line(3,2){10}}
    \put(285,63.4){\line(1,0){15}}
    \put(285,63.4){\line(0,1){15}}
    \put(270,78.4){\line(1,0){15}}
    \put(250,50){\line(3,2){10}}
    \put(250,50){\line(1,0){15}}
    \put(265,50){\line(3,2){20}}
    \put(260,56.7){\line(0,1){15}}
    \put(260,71.7){\line(1,0){15}}
    \put(260,71.7){\line(3,2){10}}
    \put(275,71.7){\line(3,2){10}}
  \end{picture}
\caption{C4 sums of 2-dimensional standard sets yielding a 3-dimensional standard set}
\phantomsection\label{figure:C4sums}
\end{figure}
\end{center}

It is easy to see that
\begin{itemize}
  \item $\Delta_1 + \Delta_2$ is a standard set;
  \item its cardinality is the sum of the cardinalities of $\Delta_1$ and $\Delta_2$; 
  \item C4 addition is associative and commutative, and $\emptyset$ is its neutral element; 
  \item $G(\Delta_1 + \Delta_2) = \gadd{G(\Delta_1)}{G(\Delta_2)}$. 
\end{itemize}
The last item confirms that C4 addition of standard set is indeed the counterpart of addition of standard graphs. 
Here is the counterpart of standard decomposition of standard graphs. 

\begin{dfn}[C4 decomposition]
\phantomsection\label{dfn:C4Decomposition}
Let $\Delta \subseteq \N^d$ be a finite standard set. 
A \emph{C4 decomposition of $\Delta$} is a multiset $\set{\Delta_1, \ldots, \Delta_h}$ 
of standard sets in $\N^{d - 1}$ whose C4 sum equals $\Delta$. 
Here we understand each $\Delta_i$ to be a standard set in $\N^d$ via the embedding 
$\N^{d-1} \hookrightarrow \N^d : \gamma \mapsto (\gamma,0)$. 
\end{dfn}

Figure \ref{figure:C4sums} shows C4 decompositions of the standard set in $\N^3$ on the right hand side
into two (multi)sets of standard set in $\N^2$. 
Note, however, that the three-dimensional standard set of that example 
has more C4 decompositions than the two shown in the figure. 

The following proposition is the first step of four in proving that
C4 decomposition and standard decomposition of labeled graphs are
equivalent.

\begin{pro}
\phantomsection\label{pro:C4vsGraph}
  Let $\Delta \subseteq \N^d$ be a finite standard set. 
  Then the C4 decompositions of $\Delta$ and the standard decompositions of $G(\Delta)$ 
  are in canonical bijection. 
\end{pro}

\begin{proof}
  Let $\set{\Delta_1, \ldots, \Delta_h}$ be a C4 decomposition of $\Delta$. 
  Consider, for $j = 1, \ldots, h$, the graph $H_j$ whose nodes and edges are identical 
  to the nodes and edges of $G(\Delta)$ and whose labeling is given by 
  \[
    \lab{H_j} \gamma = 
    \begin{cases}
      1 \text{ if } \gamma \in H_j \\
      0 \text{ else}.
    \end{cases}
  \]
  In other words, we think of $\Delta_j$, which is a priori a standard
  set in $\N^{d-1}$, as being a standard set in $\N^d$, as we do in
  Definition \ref{dfn:C4Decomposition}, and define $H_j \defeq
  G(\Delta_j)$.  Then $H_j$ is obviously a standard 0-1 graph.  The
  fact that $\set{\Delta_1, \ldots, \Delta_h}$ is a C4 decomposition
  of $\Delta$ implies that $\hsym \defeq \set{H_1, \ldots, H_h}$ is a
  standard decomposition of $G(\Delta)$.
  
  Conversely, let $\hsym$ be a standard decomposition of $G(\Delta)$. 
  Recall that the node set of $G(\Delta)$ is $\Delta\p \defeq q^d(\Delta)$, which is a standard set in $\N^{d-1}$. 
  For every $H \in \hsym$, we define $\Delta(H)$ to be the set of all $\gamma \in \Delta\p$
  with $\lab{H} \gamma = 1$. 
  The definition of $\edges {G(\Delta)}$, together with the fact that $H$ is a standard graph, 
  shows that $\Delta(H) \subseteq \N^{d - 1}$ is a standard set contained in $\Delta\p$. 
  The fact that $\hsym$ is a standard decomposition of $G(\Delta)$ 
  means that for each $\gamma \in \Delta\p$, the labels of all nodes $\gamma$, which are 0 or 1, 
  sum up to the height $h_\gamma$. 
  This means that C4 sum of the corresponding multiset 
  $\setBuilder{\Delta(H)}{H \in \hsym}$ equals $\Delta$, so that multiset is a C4 decomposition of $\Delta$. 
  
  The two constructions are readily seen to be mutual inverses. 
\end{proof}


\section{Canonicalization for graphs of standard sets}
\phantomsection\label{sec:canonicalization}

The graph of a given standard set will in general contain many nodes of identical label
which are connected by an edge. 
From the discussion at the end of Section \ref{sec:standardGraphsAndComponents}, 
we know that edges between nodes of the same label are irrelevant for computing the standard decomposition of that graph. 
We also know that we can get rid of those redundancies, 
without spoiling standard decompositions, by passing from a graph to its canonicalization. 
Given a standard set $\Delta$, we should therefore not work with its standard graph, 
but rather the canonicalization of its standard graph. 

\begin{dfn}[Canonicalization of the standard graph of $\Delta$]
\phantomsection\label{dfn:canonicalization}\ \\
  \begin{itemize}
    \item
    We say that a subset $B$ of $\N^{d-1}$ is \emph{connected} if for all $\gamma,\gamma\p \in B$, 
    there exists a sequence $(\gamma_j)$ in $B$ starting at $\gamma_0 = \gamma$ and
    ending at $\gamma_n = \gamma\p$ such that for all $j$, 
    we either have $\gamma_{j+1} = \gamma_j + e_i$ or $\gamma_j = \gamma_{j+1} + e_i$ 
    for some $i \in \set{1,\ldots,d-1}$. 
    \item
    A {\it connected component} of $A \subseteq \N^{d-1}$ is a connected $B \subseteq A$, 
    maximal with respect to inclusion. 
    \item
    Let $\Delta \subseteq \N^d$ be a standard set, $h \defeq \max(q_d(\Delta))$ its height, 
    and $\Delta\p \defeq q^d(\Delta)$ its projection. 
    For $a=1, \ldots, h$, we define the \emph{$a$-th isohypse} as 
    \[
      \Delta^a \defeq \setBuilder{\gamma \in \Delta\p}{\card{(q^d)^{-1}(\gamma) \cap \Delta} = a}, 
    \]
    the set of all points in the projection of height $a$. 
    \item
    We define the graph $G\p(\Delta)$ by 
    \[
      \begin{split}
        \nodes{G\p(\Delta)} & \defeq \setBuilder{ \text{connected components of } \Delta^a}
        { a = 1, \ldots, h }; \\
        \edges{G\p(\Delta)} & \defeq \setBuilder{ (\Delta^a_b, \Delta^c_d)}
        {\exists \gamma\p \in \Delta^a_b, \gamma \in \Delta^c_d : \gamma\p = \gamma + e_i \text{ for some } i } \\
        \lab{G\p(\Delta)} {\Delta^a_b} & \defeq a, 
      \end{split}
    \]
    where we denote by $\Delta^a_b$ the connected components of isohypse $\Delta^a$. 
  \end{itemize}
\end{dfn}

The transition from $\Delta$ to $G(\Delta)$ and to $G\p(\Delta)$ is illustrated in Figure \ref{figure:standardSetsAndGraphs}. 

\begin{center}
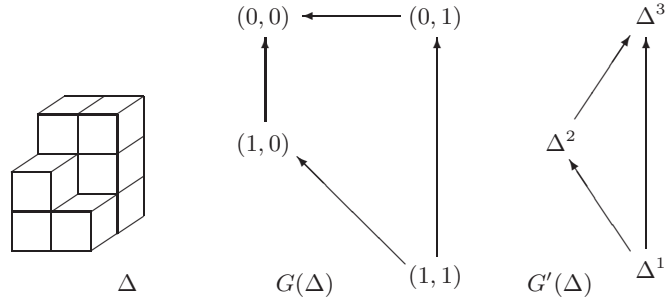
\begin{figure}[ht]
  \begin{picture}(270,130)
    \put(10,20){\line(1,0){30}}
    \put(10,20){\line(0,1){30}}
    \put(10,35){\line(1,0){30}}
    \put(10,50){\line(1,0){15}}
    \put(25,20){\line(0,1){30}}
    \put(40,20){\line(0,1){15}}
    \put(10,50){\line(3,2){10}}
    \put(25,50){\line(3,2){10}}
    \put(25,35){\line(3,2){10}}
    \put(40,35){\line(3,2){20}}
    \put(40,20){\line(3,2){20}}
    \put(20,56.8){\line(0,1){15}}
    \put(35,41.8){\line(0,1){30}}
    \put(50,26.8){\line(0,1){45}}
    \put(20,71.8){\line(1,0){30}}
    \put(20,56.8){\line(1,0){30}}
    \put(35,41.8){\line(1,0){15}}
    \put(20,71.8){\line(3,2){10}}
    \put(35,71.8){\line(3,2){10}}
    \put(50,71.8){\line(3,2){10}}
    \put(50,56.8){\line(3,2){10}}
    \put(60,78.6){\line(-1,0){30}}
    \put(60,78.6){\line(0,-1){45}}
    \put(50,5){\small $\Delta$}
    \put(95,106){\small $(0,0)$}
    \put(160,106){\small $(0,1)$}
    \put(95,58){\small $(1,0)$}
    \put(160,8){\small $(1,1)$}
    \put(157,109){\vector(-1,0){37}}
    \put(106,69){\vector(0,1){32}}
    \put(171,18){\vector(0,1){83}}
    \put(159,15){\vector(-1,1){41}}
    \put(110,5){\small $G(\Delta)$}
    \put(246,106){\small $\Delta^3$}
    \put(212,58){\small $\Delta^2$}
    \put(246,10){\small $\Delta^1$}
    \put(223,70){\vector(2,3){21}}
    \put(250,21){\vector(0,1){80}}
    \put(244,20){\vector(-2,3){23}}
    \put(205,5){\small $G\p(\Delta)$}
  \end{picture}
\caption{A standard set of height 3, its graph, and its canonicalized graph}
\phantomsection\label{figure:standardSetsAndGraphs}
\end{figure}
\end{center}

The following proposition is the second step of four in proving that
C4 decomposition and standard decomposition of labeled graphs are
equivalent.

\begin{pro}
\phantomsection\label{pro:canonicalization}
  Let $\Delta \subseteq \N^d$ be a finite standard set.  Then
  $G\p(\Delta)$, as defined above, is the canonicalization of the standard graph of $\Delta$.
\end{pro}

\begin{proof}
  Let $G(\Delta)$ be the standard graph of $\Delta$. 
  Then $\Delta\p$ is the node set of $G(\Delta)$, and two nodes get the same label if, and only if, 
  they lie in the same isohypse $\Delta^a$. 
  Moreover, the definition of $G(\Delta)$ shows that two nodes of this graph lying in the same isohypse
  are connected by a sequence of edges in that graph if, and only if, 
  they lie in the same connected component of some $\Delta^a$.
  So we may contract each connected component $\Delta^a_b$ of each isohypse to one node. 
  This is what the definition of $G\p(\Delta)$ does. 
  
  For finishing the proof, we have to show that no more pairs of nodes in $G\p(\Delta)$ may be contracted into one node. 
  Contraction only happens if two nodes have the same label and are connected by an edge. 
  Suppose that $\Delta^a_b$ and $\Delta^a_d$ are connected by an edge. 
  Then there exist $\gamma\p \in \Delta^a_b$ and $\gamma \in \Delta^a_d$ such that 
  $\gamma\p = \gamma + e_i$, so $\gamma\p$ and $\gamma$ lie in the same connected component of $\Delta^a$, 
  a contradiction. 
\end{proof}


\section{From standard graphs with unique maximal nodes to standard sets}

For each standard set $\Delta$, the canonicalized graph $G\p(\Delta)$ is connected
and contains a unique node of maximal label, namely, the highest isohypse $\Delta^h$. 
This graph thus lies in the class $\mathcal{S}$ defined in the Introduction. 
Example \ref{ex:graphArisingFromSet} and Proposition \ref{pro:graphNotArisingFromSet} show that 
graphs in $\mathcal{S}$ may or may not arise from standard sets. 

\begin{ex}
\phantomsection\label{ex:graphArisingFromSet}
  Figure \ref{figure:lotsOfArrows} shows a standard graph which arises as the standard graph
  of a standard set in $\N^4$, namely, 
  \[
    \Delta = \left\{
    \begin{array}{c}
      (0,0,0,0), (0,0,0,1), (0,0,0,2), \\
      (1,0,0,0), (1,0,0,1), (0,1,0,0), (0,1,0,1), \\
      (0,0,1,0), (0,0,1,1), (0,1,1,0), (0,1,1,1), \\
      (1,1,0,0), (1,0,1,0)
    \end{array}
    \right\} .
  \]
  The picture on the right hand side of that figure shows $\Delta^3, \Delta^2$ and $\Delta^1 \subseteq \N^3$. 
\end{ex}

\begin{center}
\begin{figure}[ht]
  \begin{picture}(320,110)
    \put(47.5,102){\small 3}
    \put(24,66){\vector(2,3){22}}
    \put(76,66){\vector(-2,3){22}}
    \put(21,55){\small 2}
    \put(75,55){\small 2}
    \put(20,18){\vector(0,1){32}}
    \put(74,18){\vector(-3,2){48}}
    \put(26,18){\vector(3,2){48}}
    \put(80,18){\vector(0,1){32}}
    \put(20,7){\small 1}
    \put(76,7){\small 1}
    \put(130,40){\line(1,0){15}}
    \put(130,40){\line(0,1){15}}
    \put(130,55){\line(1,0){15}}
    \put(145,40){\line(0,1){15}}
    \put(130,55){\line(3,2){10}}
    \put(145,40){\line(3,2){10}}
    \put(145,55){\line(3,2){10}}
    \put(140,61.8){\line(1,0){15}}
    \put(155,46.8){\line(0,1){15}}
    \put(134,20){\small $\Delta^3$}
    \put(190,33.2){\line(1,0){15}}
    \put(190,33.2){\line(0,1){15}}
    \put(190,48.2){\line(1,0){15}}
    \put(205,33.2){\line(0,1){15}}
    \put(190,48.2){\line(3,2){10}}
    \put(205,33.2){\line(3,2){10}}
    \put(205,48.2){\line(3,2){10}}
    \put(200,55){\line(1,0){15}}
    \put(215,40){\line(0,1){15}}
    \put(200,55){\line(0,1){15}}
    \put(215,40){\line(1,0){15}}
    \put(200,70){\line(1,0){30}}
    \put(230,40){\line(0,1){30}}
    \put(200,70){\line(3,2){10}}
    \put(230,70){\line(3,2){10}}
    \put(230,40){\line(3,2){10}}
    \put(210,76.8){\line(1,0){30}}
    \put(240,46.8){\line(0,1){30}}
    \put(204,20){\small $\Delta^2$}
    \put(260,48.2){\line(1,0){15}}
    \put(260,48.2){\line(0,1){15}}
    \put(260,63.2){\line(1,0){15}}
    \put(275,48.2){\line(0,1){15}}
    \put(260,63.2){\line(3,2){10}}
    \put(275,48.2){\line(3,2){10}}
    \put(275,63.2){\line(3,2){10}}
    \put(270,70){\line(1,0){15}}
    \put(285,55){\line(0,1){15}}
    \put(275,33.2){\line(1,0){15}}
    \put(275,33.2){\line(0,1){15}}
    \put(275,48.2){\line(1,0){15}}
    \put(290,33.2){\line(0,1){15}}
    \put(290,33.2){\line(3,2){10}}
    \put(290,48.2){\line(3,2){10}}
    \put(285,55){\line(1,0){15}}
    \put(300,40){\line(0,1){15}}
    \put(274,20){\small $\Delta^1$}
  \end{picture}
\caption{A standard graph arising from a standard set in $\N^4$}
\phantomsection\label{figure:lotsOfArrows}
\end{figure}
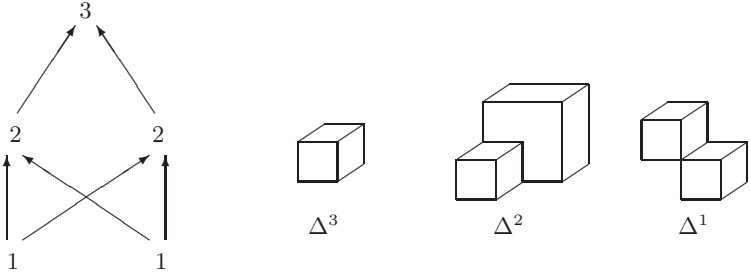
\end{center}

\begin{pro}
\phantomsection\label{pro:graphNotArisingFromSet}
  The graph shown in Figure \ref{figure:badArrows}
  does not arise as the standard graph of a standard set. 
\end{pro}

\begin{center}
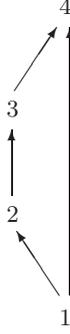
\begin{figure}[ht]
  \begin{picture}(34,130)
    \put(30,119){\small $4$}
    \put(10,80){\small $3$}
    \put(10,40){\small $2$}
    \put(30,1){\small $1$}
    \put(34,12.5){\vector(0,1){101.5}}
    \put(12,50){\vector(0,1){25}}
    \put(30,12.5){\vector(-2,3){16}}
    \put(13,90){\vector(2,3){16}}
  \end{picture}
\caption{A graph not arising from a standard set}
\phantomsection\label{figure:badArrows}
\end{figure}
\end{center}

\begin{proof}
  Assume that $\Delta \subseteq \N^d$ is a standard set whose standard graph is the given graph $G$. 
  In particular, the nodes of $G$ are the isohypses $\Delta^i$, for $i=1,2,3,4$. 
  We claim that there exists an element $\beta \in \Delta^1$
  and $i, j \in \set{1, \ldots, d-1}$ such that $\beta - e_i \in \Delta^2$ and $\beta - e_j \in \Delta^4$. 
  This will finish the proof, since $\beta - e_i - e_j$ will then lie in $\Delta$. 
  But $\beta - e_i - e_j$ can lie in neither $\Delta^1$ nor $\Delta^2$ nor $\Delta^3$, 
  since either of these inclusions would contradict the standard set property of $\Delta$. 
  However, an inclusion $\beta - e_i - e_j \in \Delta^4$ 
  would force an edge from node $\Delta^2$ to node $\Delta^4$ in the standard graph of $\Delta$, which isn't there. 
  
  So we have to prove the above assertion. 
  There exists elements $\sigma \in \Delta^4$ and $\tau \in \Delta^2$ and a sequence $(\gamma_k)_{k=0}^N$ such that 
  \begin{itemize}
    \item its subsequence $(\gamma_k)_{k=1}^{N-1}$ lies in $\Delta^1$, 
    \item its starting point $\gamma_0$ is $\sigma$, 
    \item its end point $\gamma_N$ is $\tau$, and 
    \item it has the property that for all $k$, $\gamma_{k+1} = \gamma_k \pm e_i$ for some $i$. 
  \end{itemize}
  Take $\sigma$, $\tau$ and $(\gamma_k)$ sharing these properties 
  such that, in addition, $N$, the length of the sequence $(\gamma_k)$ is minimal. 
  If $N = 2$, then $\beta \defeq \gamma_1$ is of the desired shape. 
  We now assume that $N>2$, and are going to show that this assumption leads to a contradiction. 
  For doing so, we prove three claims concerning the sequence $(\gamma_k)$. 
  The first claim is that for all $k < N$,
  \begin{equation}
  \phantomsection\label{eqn:gamma_k}
    \gamma_k = \sigma + \sum_{i \in I_k} e_i
  \end{equation}
  for some multiset of indices $I_k$. Note that $\gamma_k \in \Delta^1$ for all $k$ in question. 
  For $k=0,1$, equation \eqref{eqn:gamma_k} is evident. 
  We assume that the equation holds for $k$ and prove it to hold for $k+1$. 
  Suppose that $\gamma_{k+1} = \sigma + \sum_{i \in I_k} e_i - e_j$ for some $j \notin I_k$. 
  Then, in particular, $\sigma\p \defeq \sigma - e_j \in \Delta^4$ and $\gamma_{k+1} \in \Delta^1$. 
  Consider the sequence $(\gamma_l\p)_{l=0}^{N-1}$, where 
  \[
    \gamma_l\p \defeq 
    \begin{cases}
    \gamma_l - e_j & \text{ for } l < k, \\
    \gamma_{l+1} & \text{ for }  l \geq k. 
    \end{cases}
  \]
  This sequence is one element shorter than the original sequence $(\gamma_k)$. 
  Like the original sequence, it starts in $\Delta^4$ and ends in $\Delta^2$. 
  A priori the elements $\gamma\p_m$, for $m = 1, \ldots, k-1$, 
  may lie in $\Delta^1$, $\Delta^2$, $\Delta^3$ or $\Delta^4$. 
  \begin{itemize}
    \item If all of them lie in $\Delta^1$, the sequence $(\gamma\p_l)$ contradicts the minimality of $N$. 
    \item If $\gamma\p_m \in \Delta^2$, the sequence $(\gamma\pp_p)_{p = 0}^{m+1}$, where
    \[
      \gamma_p\pp \defeq 
      \begin{cases}
      \gamma_p & \text{ for } l \leq m, \\
      \gamma\p_m & \text{ for }  p = m+1
      \end{cases}
    \]
    contradicts the minimality of $N$.
    \item If $\gamma\p_m \in \Delta^3$, we obtain an edge from node $\Delta^1$ to node $\Delta^3$, which isn't there. 
    \item If $\gamma\p_m \in \Delta^4$, we consider the
      largest index $M$ such that $\gamma\p_M\in \Delta^4$ and consider
      the subsquence $\gamma
      \p_M,\dots,\gamma\p_N$. The condition $\gamma_i \in \Delta^1$
      implies that $\gamma\p_i\in \Delta^1 \cup \Delta^2 \cup \Delta^4$, since there is no
      node from $\Delta^1$ to $\Delta^3$. Thus the first term of  $\gamma
      \p_M,\dots,\gamma\p_N$ lies in $\Delta^4$, and the other terms 
      in $\Delta^1\cup \Delta^2$. The
      subsequence  $\gamma
      \p_M,\dots,\gamma\p_{M\p}$, where $M\p\geq M$ is the smallest index 
      with $\gamma\p_{M\p}\in \Delta^2$, contradicts the minimality of
      $N$. 
  \end{itemize}
  This finishes the proof of the first claim. 
  
  Our second claim is that $I_k \subseteq I_{k+1}$ for all sets appearing in \eqref{eqn:gamma_k}. 
  This is true for $I_0 \subseteq I_1$; moreover, since $\gamma_{k+1} = \gamma_k \pm e_i$, 
  our first claim shows that either $I_k \subseteq I_{k+1}$ or $I_k \supseteq I_{k+1}$ holds. 
  Let $m$ be the smallest index such that $I_m \supseteq I_{m+1}$. 
  Then the sequence $(\gamma_k)_{k=0}^{m+1}$ is obtained by adding to $\sigma$
  a number of $e_i$, one by one, and finally subtracting one of them, say $e_j$. 
  We obtain a shorter sequence $(\gamma\p_k)_{k=0}^{m-1}$ by adding to $\sigma$ the same sequence of $e_i$ as above, 
  but leaving out $e_j$. 
  The same arguments as the ones from the four bulleted items above then lead to a contradiction. 
  This finishes the proof of the second claim. 
  
  Our third claim is that $\tau$, the final member of our sequence $(\gamma_k)$, takes the shape 
  \[
    \tau = \gamma_N = \gamma_{N-1} - e_j ,
  \]
  for some $e_j \notin I_{N-1}$. 
  The complementary cases include $\gamma_N = \gamma_{N-1} - e_j$ for some $e_j \in I_{N-1}$, 
  which immediately contradicts minimality of $N$, and $\gamma_N = \gamma_{N-1} + e_j$ for some $e_j$. 
  In the latter case, the inclusion $\gamma_N = \tau \in \Delta^2$ shows that $\gamma_{N-1}$ would also lie in $\Delta^2$, 
  a contradiction. 
  This finishes the proof of the third claim. 
  
  The sequence $(\gamma_k)_{k=0}^N$ is therefore obtained by adding to $\sigma$
  a number of $e_i$, one by one, and finally subtracting some $e_j$ which is not found among the $e_i$ previously added. 
  We denote by $e_u$ the last element from the sequence of $e_i$ which we add, 
  that is, the one element which we add for passing from $\gamma_{N-2}$ to $\gamma_{N-1}$. 
  Consider $\rho \defeq \tau - e_u$. 
  Then $\rho$ may lie in $\Delta^2$, $\Delta^3$ or $\Delta^4$. 
  \begin{itemize}
    \item If $\tau\p \in \Delta^2$, the sequence $(\gamma\p_l)_{l = 0}^{N-1}$, where
    \[
      \gamma_l\p \defeq 
      \begin{cases}
      \gamma_l & \text{ for } l \leq N-2, \\
      \tau\p & \text{ for }  l = N-1
      \end{cases}
    \]
    contradicts the minimality of $N$.
    \item If $\tau\p \in \Delta^3$, we obtain an edge from node $\Delta^1$ to node $\Delta^3$, which isn't there. 
    \item If $\tau\p \in \Delta^4$, we obtain an edge from node $\Delta^2$ to node $\Delta^4$, which isn't there. 
  \end{itemize}
  So we have disproved the assumption that $N>2$. The proposition follows. 
\end{proof}

The graphs in Figures \ref{figure:lotsOfArrows} and
\ref{figure:badArrows} define relations on their respective node sets
which both fail to be transitive.  So one might not guess that
transitivity of graphs in $\mathcal{S}$ is crucial for such graphs to
arise from standard sets.  That, however, is indeed true, as we shall
see in Proposition \ref{pro:graphToStdSet} below.  Let us first
establish that passing from a graph to its transitive closure has no
impact on standard decompositions.

\begin{lmm}
\phantomsection\label{lmm:transitiveClosure}
  Let $G$ be a standard graph and $\overline{G}$ its transitive closure. 
  Then the standard decompositions of $G$ and $\overline{G}$ are in canonical bijection. 
\end{lmm}

\begin{proof}
  Given a standard decomposition $\hsym$ of $G$, 
  replace every member $H$ by its transitive closure $\overline{H}$. 
  The resulting multiset $\overline{\hsym}$ is a standard decomposition of $\overline{G}$. 
  Given a standard decomposition $\mathcal{K}$ of $\overline{G}$, 
  we delete from every member $K$ all edges that appear in $\overline{G}$ but not in $G$, 
  and call the resulting graph $K^\circ$. The resulting multiset $\mathcal{K}^\circ$ is a standard decomposition of $G$. 
  The maps $\hsym \mapsto \overline{\hsym}$ and $\mathcal{K} \mapsto \mathcal{K}^\circ$
  are mutual inverses. 
\end{proof}

The following proposition is the third step of four in
proving that C4 decomposition and standard decomposition of labeled
graphs are equivalent.

\begin{pro}
\phantomsection\label{pro:graphToStdSet}
  Let $G$ be a canonical, connected and transitive standard graph containing a unique node of maximal label. 
  Then there exists a standard set $\Delta \subseteq \N^d$, for some $d \geq 1$, 
  whose canonicalized standard graph is $G$. 
\end{pro}

\begin{proof}
  Upon using the terminology of Definition \ref{dfn:canonicalization}, we
  denote by $G\p(\Delta)$ the canonicalized standard graph of a
  standard set $\Delta$.  We prove the proposition by two nested
  inductions, the outer over the number of nodes of $G$, and the inner
  over the number of edges of $G$.  The base case of the outer
  induction is trivial.  As for the outer induction step, let $G$ be a
  given connected and transitive standard graph containing a unique
  node $v_h$ of maximal label, $h$. 
  Let $v_0$ be a node of minimal label. 
  We remove from $G$ the node $v_0$, 
  along with all edges whose source is $v_0$.
  We call the graph thus obtained $G_0$. 
  Then $G_0$ is also canonical, connected and transitive. 
  Canonicity and transitivity are obvious. As for connectedness, 
  we note that each node in $G$ other than the node $v_0$ is the starting point of a 
  sequence of edges ending up in $v_h$, 
  which sequence does not pass through $v_0$ by minimality of $v_0$ 
  and canonicity of $G$. 
  Moreover, when replacing $G$ by $G_0$, 
  we do not change the labels of the remaining nodes. 
  Thus $G_0$ contains a unique node of maximal label. We may therefore assume
  that there exists a standard set $\Delta_0 \subseteq \N^d$, for some
  $d$, such that $G\p(\Delta_0) = G_0$.
  
  For establishing the outer induction step, we shall put the node $v_0$ back into the graph. 
  Transitivity of $G$ implies that this graph contains an edge from $v_0$ to $v_h$. 
  Let $G_1$ be the (transitive) graph that arises from $G_0$ by adding the one node $v_0$ and the one edge $(v_0,v_h)$. 
  We now construct a standard set $\Delta_1$ such that $G\p(\Delta_1) = G_1$. 

  Consider the embedding 
  $\iota: \N^d \hookrightarrow \N^{d + 1} : \beta \mapsto (0, \beta)$.
  The transition from $\Delta_0$ to $\iota(\Delta_0)$ does not affect the standard graph of $\Delta_0$. 
  We may therefore assume that $\Delta_0 \subseteq \N^d$ is contained in the hyperplane $\set{\beta_1 = 0}$ of $\N^d$. 
  The node $v_h \in G_0$ corresponds to the isohypse $(\Delta_0)^h$. 
  Let $h_0 < h$ be the label of $v_0$. We may assume that $v_0 > 1$. The set 
  \begin{equation}
  \phantomsection\label{eqn:Delta0ToDelta1}
    \begin{split}
      \Delta_1 & \defeq \Delta_0 \cup M_1, \text{ where }\\
      M_1 & \defeq \setBuilder{(1,0,\ldots,0,\beta_d)}{0\leq \beta_d \leq h_0-1}
    \end{split}
  \end{equation}
  is standard. 
  See Figure \ref{figure:pilar} for a visualization of the transition from $\Delta_0$ to $\Delta_1$. 
  For $a \neq h_0$, the isohypses $(\Delta_0)^a$ and $(\Delta_1)^a$ are identical. 
  The isohypse $(\Delta_1)^{h_0}$ is $(\Delta_0)^{h_0}\cup 
q^d(M_1) =(\Delta_0)^{h_0}\cup \set{e_1}$. 
  When passing to $G\p(\Delta_1)$, we see that this graph arises from $G\p(\Delta_0)$ 
  by adding the one node $ q^d(M_1) $ and the one edge connecting that new node and $(\Delta_1)^h$. 
  This establishes the outer induction step, and at the same time the inner induction basis. 
  
  As for the inner induction step, we may assume to have a transitive graph $G_1$
  \begin{itemize}
    \item with the same nodes and the same labels as $G$, 
    \item and a distinguished node $v_0$ 
    \item such that all edges but those with source $v_0$ agree in $G$ and $G_1$,
  \end{itemize}
  along with a standard set $\Delta_1 \subseteq \N^d$ such that $G\p(\Delta_1) = G_1$. 
  Let $v_1$ be a node of $G$ such that $(v_0, v_1)$ is an edge in $G$, 
  but our original graph $G$ contains no chain of edges from $v_0$ to $v_1$ of length more than 1. 
  We may assume that $\lab G {v_0} < \lab G {v_1}$.
  Denote by $G_2$ the graph that arises from $G_1$ by adding the edge $(v_0, v_1)$. 
  We will prove the existence of a standard set $\Delta_2$ such that $G\p(\Delta_2) = G_2$. 
  This will establish the inner induction step, and finish the proof of the proposition. 
  
  Analogously as above, we assume that $\Delta_1 \subseteq \N^d$ is contained in 
  the hyperplane $\set{\beta_1 = 0}$ of $\N^d$. 
  The choice of $v_1$ implies that $G_2$ is again transitive. 
  For $i = 0,1$, the node $v_i \in G_1$ corresponds to a connected component $C_i$ of $(\Delta_1)^{h_i}$, 
  where $h_i$ is the label of $v_i$. 
  The set 
  \begin{equation}
  \phantomsection\label{eqn:Delta1ToDelta3/2}
    \begin{split}
      \Delta_{1\frac{1}{2}} & \defeq \Delta_1 \cup M_{\frac{1}{2}}, \text{ where }\\
      M_{1\frac{1}{2}} & \defeq \Bigl(\cup_{\alpha \in \N^d}\left((q^d)^{-1}(C_1)\cap\Delta_1 + e_1 - \alpha\right)\Bigr)\cap\N^d
    \end{split}
  \end{equation}
  is standard. 
  See the first two pictures in Figure \ref{figure:mushroom}
  for a visualization of the transition from $\Delta_1$ to $\Delta_{1\frac{1}{2}}$: 
  We create a copy of the set $(q^d)^{-1}(C_1) \cap \Delta_1$ in the hyperplane $\set{\beta_1 = 1}$ of $\N^d$
  and subsequently pass to the smallest standard set containing both $\Delta_1$ and that copy. 
  Transitivity of $G_1$ implies that $G\p(\Delta_{1\frac{1}{2}}) = G\p(\Delta_1)$. 
  Indeed, for all heights $a \neq h_1$, 
  the connected components of $(\Delta_{1\frac{1}{2}})^a$ are identical to 
  of the connected components of $(\Delta_1)^a$. 
  For height $h_1$, the same is true for those connected components of $(\Delta_{1\frac{1}{2}})^{h_1}$ 
  that do not project to $C_1$ under $q^d$. 
  The connected component $C_1$ of $(\Delta_1)^{h_1}$, however, has a much larger counterpart in $\Delta_{1\frac{1}{2}}$, 
  namely, the union of $C_1$ and the set $q^d(M_{1\frac{1}{2}})$. 
  As for edges in $G\p(\Delta_{1\frac{1}{2}})$ emerging from node $C_1 \cup q^d(M_{1\frac{1}{2}})$, 
  the presence of $M_{1\frac{1}{2}}$ obviously leads to new adjacencies in connected components 
  of isohypses of $\Delta_{1 \frac{1}{2}}$. 
  But transitivity of $G_1$ guarantees that none of those adjacencies lead to an edge in $G\p(\Delta_{1\frac{1}{2}})$
  that does exist in $G\p(\Delta_1)$. So the graphs $G\p(\Delta_1)$ and $G\p(\Delta_{1\frac{1}{2}})$ are identical. 
  
  However, we do not want another standard set with the same canonicalized graph, 
  but rather a graph with one additional edge. 
  We obtain that edge by applying the same trick once more, defining 
  \begin{equation}
  \phantomsection\label{eqn:Delta3/2ToDelta2}
    \begin{split}
      \Delta_2 & \defeq \Delta_1 \cup M_{1\frac{1}{2}} \cup M_2, \text{ where }\\
      M_2 & \defeq \Bigl(\cup_{\alpha \in \N^d}\left((q^d)^{-1}(C_1) + e_1 - \alpha\right)\Bigr)\cap\N^d.
    \end{split}
  \end{equation}
  This is another standard set. 
  See the last two pictures in Figure \ref{figure:mushroom}
  for a visualization of the transition from $\Delta_{1\frac{1}{2}}$ to $\Delta_2$:
  We also create a copy of the set $(q^d)^{-1}(C_1) \cap \Delta_1$ in the hyperplane $\set{\beta_1 = 1}$ of $\N^d$
  and subsequently pass to the smallest standard set containing both $\Delta_1$ and that copy. 
  For all heights $a \neq h_0, h_1$, 
  the connected components of $(\Delta_2)^a$ are identical to
  the connected components of $(\Delta_{1\frac{1}{2}})^a$. 
  For heights $a = h_0, h_1$, the same is true for those connected components of $(\Delta_2)^a$ 
  that do not project to $C_0$ or $C_1$. 
  Note that the sets $M_{1\frac{1}{2}}$ and $M_2$ will in general intersect. 
  The counterpart of $C_1$ in $\Delta_2$ is the union $C_1 \cup M_{1\frac{1}{2}}$; 
  and the counterpart of $C_0$ in $\Delta_2$ is $(C_0 \cup M_2) \setminus M_{1\frac{1}{2}}$. 
  The graph $G\p(\Delta_2)$ contains all the edges that appear in $G\p(\Delta_{1\frac{1}{2}})$, 
  plus an edge from node $(C_0 \cup M_2) \setminus M_{1\frac{1}{2}}$
  to node $C_1 \cup M_{1\frac{1}{2}}$ : the extra edge exists since
  $(1,0,\dots,0,h_1)\in C_1\cup M_{1\frac{1}{2}}$ and $z_0+e_1\in
  (C_0 \cup M_2) \setminus M_{1\frac{1}{2}}$ for $z_0 \in (q^d)^{-1}(C_0)$.
  This establishes the inner induction step. 
\end{proof}

\begin{center}
\begin{figure}[ht]
  \begin{picture}(300,150)
    \put(20,16.8){\line(1,0){90}}
    \put(20,16.8){\line(0,1){120}}
    \put(20,136.8){\line(1,0){30}}
    \put(30,143.6){\line(1,0){30}}
    \put(20,136.8){\line(3,2){10}}
    \put(50,136.8){\line(3,2){10}}
    \put(50,106.8){\line(1,0){15}}
    \put(60,113.6){\line(1,0){15}}
    \put(50,106.8){\line(3,2){10}}
    \put(65,106.8){\line(3,2){10}}
    \put(50,106.8){\line(0,1){30}}
    \put(60,113.6){\line(0,1){30}}
    \put(65,91.8){\line(1,0){15}}
    \put(75,98.6){\line(1,0){15}}
    \put(65,91.8){\line(3,2){10}}
    \put(80,91.8){\line(3,2){10}}
    \put(65,91.8){\line(0,1){15}}
    \put(75,98.6){\line(0,1){15}}
    \put(80,61.8){\line(1,0){15}}
    \put(90,68.6){\line(1,0){15}}
    \put(80,61.8){\line(3,2){10}}
    \put(95,61.8){\line(3,2){10}}
    \put(80,61.8){\line(0,1){30}}
    \put(90,68.6){\line(0,1){30}}
    \put(95,31.8){\line(1,0){15}}
    \put(105,38.6){\line(1,0){15}}
    \put(95,31.8){\line(3,2){10}}
    \put(110,31.8){\line(3,2){10}}
    \put(95,31.8){\line(0,1){30}}
    \put(105,38.6){\line(0,1){30}}
    \put(110,16.8){\line(3,2){10}}
    \put(110,16.8){\line(0,1){15}}
    \put(120,23.6){\line(0,1){15}}
    \put(175,16.8){\line(1,0){75}}
    \put(175,16.8){\line(0,1){60}}
    \put(160,76.8){\line(0,1){60}}
    \put(150,10){\line(0,1){60}}
    \put(165,10){\line(0,1){60}}
    \put(150,10){\line(1,0){15}}
    \put(150,70){\line(1,0){15}}
    \put(160,76.8){\line(1,0){15}}
    \put(165,10){\line(3,2){10}}
    \put(150,70){\line(3,2){10}}
    \put(165,70){\line(3,2){10}}
    \put(160,136.8){\line(1,0){30}}
    \put(170,143.6){\line(1,0){30}}
    \put(160,136.8){\line(3,2){10}}
    \put(190,136.8){\line(3,2){10}}
    \put(190,106.8){\line(1,0){15}}
    \put(200,113.6){\line(1,0){15}}
    \put(190,106.8){\line(3,2){10}}
    \put(205,106.8){\line(3,2){10}}
    \put(190,106.8){\line(0,1){30}}
    \put(200,113.6){\line(0,1){30}}
    \put(205,91.8){\line(1,0){15}}
    \put(215,98.6){\line(1,0){15}}
    \put(205,91.8){\line(3,2){10}}
    \put(220,91.8){\line(3,2){10}}
    \put(205,91.8){\line(0,1){15}}
    \put(215,98.6){\line(0,1){15}}
    \put(220,61.8){\line(1,0){15}}
    \put(230,68.6){\line(1,0){15}}
    \put(220,61.8){\line(3,2){10}}
    \put(235,61.8){\line(3,2){10}}
    \put(220,61.8){\line(0,1){30}}
    \put(230,68.6){\line(0,1){30}}
    \put(235,31.8){\line(1,0){15}}
    \put(245,38.6){\line(1,0){15}}
    \put(235,31.8){\line(3,2){10}}
    \put(250,31.8){\line(3,2){10}}
    \put(235,31.8){\line(0,1){30}}
    \put(245,38.6){\line(0,1){30}}
    \put(250,16.8){\line(3,2){10}}
    \put(250,16.8){\line(0,1){15}}
    \put(260,23.6){\line(0,1){15}}
    \put(290,23.6){\line(0,1){120}}
    \put(286,23.6){\line(1,0){8}}
    \put(286,143.6){\line(1,0){4}}
    \put(290,83.6){\line(1,0){4}}
    \put(278,100){\small $h$}
    \put(294,56){\small $h_0$}
  \end{picture}
\caption{From $\Delta_0$ to $\Delta_1$}
\phantomsection\label{figure:pilar}
\end{figure}
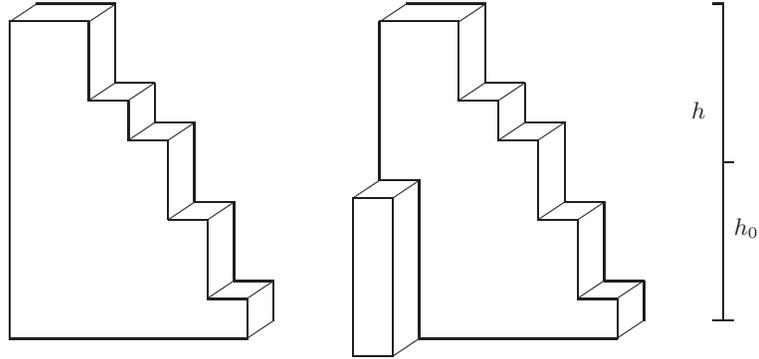
\end{center}

\begin{center}
\begin{figure}[ht]
  \begin{picture}(370,150)
    \put(0,16.8){\line(1,0){90}}
    \put(0,16.8){\line(0,1){120}}
    \put(0,136.8){\line(1,0){30}}
    \put(10,143.6){\line(1,0){30}}
    \put(0,136.8){\line(3,2){10}}
    \put(30,136.8){\line(3,2){10}}
    \put(30,106.8){\line(1,0){15}}
    \put(40,113.6){\line(1,0){15}}
    \put(30,106.8){\line(3,2){10}}
    \put(45,106.8){\line(3,2){10}}
    \put(30,106.8){\line(0,1){30}}
    \put(40,113.6){\line(0,1){30}}
    \put(45,91.8){\line(1,0){15}}
    \put(55,98.6){\line(1,0){15}}
    \put(45,91.8){\line(3,2){10}}
    \put(60,91.8){\line(3,2){10}}
    \put(45,91.8){\line(0,1){15}}
    \put(55,98.6){\line(0,1){15}}
    \put(60,61.8){\line(1,0){15}}
    \put(70,68.6){\line(1,0){15}}
    \put(60,61.8){\line(3,2){10}}
    \put(75,61.8){\line(3,2){10}}
    \put(60,61.8){\line(0,1){30}}
    \put(70,68.6){\line(0,1){30}}
    \put(75,31.8){\line(1,0){15}}
    \put(85,38.6){\line(1,0){15}}
    \put(75,31.8){\line(3,2){10}}
    \put(90,31.8){\line(3,2){10}}
    \put(75,31.8){\line(0,1){30}}
    \put(85,38.6){\line(0,1){30}}
    \put(90,16.8){\line(3,2){10}}
    \put(90,16.8){\line(0,1){15}}
    \put(100,23.6){\line(0,1){15}}
    \put(165,16.8){\line(1,0){45}}
    \put(165,16.8){\line(0,1){75}}
    \put(120,106.8){\line(0,1){30}}
    \put(110,10){\line(0,1){90}}
    \put(155,10){\line(0,1){90}}
    \put(110,10){\line(1,0){45}}
    \put(110,100){\line(1,0){45}}
    \put(120,106.8){\line(1,0){30}}
    \put(155,10){\line(3,2){10}}
    \put(110,100){\line(3,2){10}}
    \put(155,100){\line(3,2){20}}
    \put(120,136.8){\line(1,0){30}}
    \put(130,143.6){\line(1,0){30}}
    \put(120,136.8){\line(3,2){10}}
    \put(150,136.8){\line(3,2){10}}
    \put(160,113.6){\line(1,0){15}}
    \put(150,106.8){\line(3,2){10}}
    \put(150,106.8){\line(0,1){30}}
    \put(160,113.6){\line(0,1){30}}
    \put(165,91.8){\line(1,0){15}}
    \put(175,98.6){\line(1,0){15}}
    \put(165,91.8){\line(3,2){10}}
    \put(180,91.8){\line(3,2){10}}
    \put(175,98.6){\line(0,1){15}}
    \put(180,61.8){\line(1,0){15}}
    \put(190,68.6){\line(1,0){15}}
    \put(180,61.8){\line(3,2){10}}
    \put(195,61.8){\line(3,2){10}}
    \put(180,61.8){\line(0,1){30}}
    \put(190,68.6){\line(0,1){30}}
    \put(195,31.8){\line(1,0){15}}
    \put(205,38.6){\line(1,0){15}}
    \put(195,31.8){\line(3,2){10}}
    \put(210,31.8){\line(3,2){10}}
    \put(195,31.8){\line(0,1){30}}
    \put(205,38.6){\line(0,1){30}}
    \put(210,16.8){\line(3,2){10}}
    \put(210,16.8){\line(0,1){15}}
    \put(220,23.6){\line(0,1){15}}
    \put(315,16.8){\line(1,0){15}}
    \put(285,61.8){\line(0,1){30}}
    \put(240,106.8){\line(0,1){30}}
    \put(315,16.8){\line(0,1){15}}
    \put(230,10){\line(0,1){90}}
    \put(275,55){\line(0,1){45}}
    \put(305,10){\line(0,1){45}}
    \put(230,10){\line(1,0){75}}
    \put(230,100){\line(1,0){45}}
    \put(275,55){\line(1,0){30}}
    \put(240,106.8){\line(1,0){30}}
    \put(285,61.8){\line(1,0){15}}
    \put(275,55){\line(3,2){10}}
    \put(305,55){\line(3,2){20}}
    \put(305,10){\line(3,2){10}}
    \put(230,100){\line(3,2){10}}
    \put(275,100){\line(3,2){20}}
    \put(240,136.8){\line(1,0){30}}
    \put(250,143.6){\line(1,0){30}}
    \put(240,136.8){\line(3,2){10}}
    \put(270,136.8){\line(3,2){10}}
    \put(280,113.6){\line(1,0){15}}
    \put(270,106.8){\line(3,2){10}}
    \put(270,106.8){\line(0,1){30}}
    \put(280,113.6){\line(0,1){30}}
    \put(285,91.8){\line(1,0){15}}
    \put(295,98.6){\line(1,0){15}}
    \put(285,91.8){\line(3,2){10}}
    \put(300,91.8){\line(3,2){10}}
    \put(295,98.6){\line(0,1){15}}
    \put(310,68.6){\line(1,0){15}}
    \put(300,61.8){\line(3,2){10}}
    \put(300,61.8){\line(0,1){30}}
    \put(310,68.6){\line(0,1){30}}
    \put(315,31.8){\line(1,0){15}}
    \put(325,38.6){\line(1,0){15}}
    \put(315,31.8){\line(3,2){10}}
    \put(330,31.8){\line(3,2){10}}
    \put(325,38.6){\line(0,1){30}}
    \put(330,16.8){\line(3,2){10}}
    \put(330,16.8){\line(0,1){15}}
    \put(340,23.6){\line(0,1){15}}
    \put(360,23.6){\line(0,1){90}}
    \put(356,23.6){\line(1,0){8}}
    \put(356,113.6){\line(1,0){4}}
    \put(360,68.6){\line(1,0){4}}
    \put(345,80){\small $h_1$}
    \put(364,46){\small $h_0$}
  \end{picture}
\caption{From $\Delta_1$ to $\Delta_{1\frac{1}{2}}$ and $\Delta_2$}
\phantomsection\label{figure:mushroom}
\end{figure}
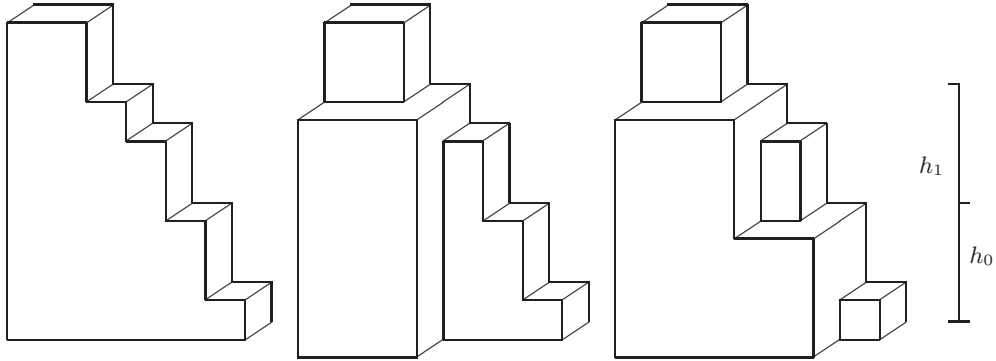
\end{center}

Readers might wonder how the polynomial dependence 
from Theorem \ref{thm:generatingComplexity} is preserved in Proposition \ref{pro:graphToStdSet}. 
Indeed, in the inductive construction of the standard set $\Delta$ from the proof of the proposition, 
the dimension of $\Delta$ and the number of elements in it grow rapidly. 
However, we don't specify $\Delta$ as list of its elements, 
but rather as a list of the minimal generators of the $\N^d$-module $\N^d \setminus \Delta$.
This set is also known as the set of \emph{outer corners} of $\Delta$.
Doing so, we avoid large data sets when handling large standard sets. 
We will use this representation of $\Delta$ in the proof of Theorem \ref{thm:graphsVsC4} below. 


\section{Reduction to standard graphs with unique maximal nodes}

The following proposition provides the fourth and last step in
proving that C4 decomposition and standard decomposition of labeled
graphs are equivalent. 
Here is a small example illustrating its assertion. 

\begin{ex}
\phantomsection\label{ex:reductionToUniqueMaximalNode}
  Let $G$ be the graph with nodes $x$ and $y$, both of label 1, and no edges. 
  Let $G\p$ be the graph with nodes $x$ and $y$ of label 1 and $z$ of label 2, 
  with edges from $x$ and from $y$ to $z$. 
  Figure \ref{figure:reductionToUniqueMaximalNode}
  shows that there is a bijection between the standard decompositions of $G$ 
  and the standard decompositions of $G\p$. 
\end{ex}

\begin{center}
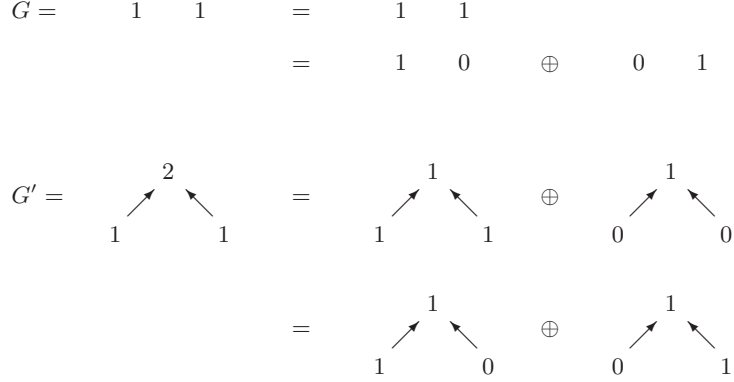
\begin{figure}[ht]
  \begin{picture}(290,155)
    \put(3,144){\small $G =$}
    \put(48,144){\small $1$}
    \put(72,144){\small $1$}
    \put(109,144){\small $=$}
    \put(148,144){\small $1$}
    \put(172,144){\small $1$}
    \put(109,124){\small $=$}
    \put(148,124){\small $1$}
    \put(172,124){\small $0$}
    \put(203,124){\small $\oplus$}    
    \put(238,124){\small $0$}
    \put(262,124){\small $1$}
    \put(3,74){\small $G\p =$}
    \put(60,83){\small $2$}
    \put(47,69){\vector(1,1){10}}
    \put(79,69){\vector(-1,1){10}}
    \put(40,59){\small $1$}
    \put(81,59){\small $1$}
    \put(109,74){\small $=$}
    \put(160,83){\small $1$}
    \put(147,69){\vector(1,1){10}}
    \put(179,69){\vector(-1,1){10}}
    \put(140,59){\small $1$}
    \put(181,59){\small $1$}
    \put(203,74){\small $\oplus$}    
    \put(250,83){\small $1$}
    \put(237,69){\vector(1,1){10}}
    \put(269,69){\vector(-1,1){10}}
    \put(230,59){\small $0$}
    \put(271,59){\small $0$}
    \put(109,24){\small $=$}
    \put(160,33){\small $1$}
    \put(147,19){\vector(1,1){10}}
    \put(179,19){\vector(-1,1){10}}
    \put(140,9){\small $1$}
    \put(181,9){\small $0$}
    \put(203,24){\small $\oplus$}    
    \put(250,33){\small $1$}
    \put(237,19){\vector(1,1){10}}
    \put(269,19){\vector(-1,1){10}}
    \put(230,9){\small $0$}
    \put(271,9){\small $1$}
  \end{picture}
\caption{The decompositions of the graphs from Example \ref{ex:reductionToUniqueMaximalNode}}
\phantomsection\label{figure:reductionToUniqueMaximalNode}
\end{figure}
\end{center}

\begin{pro}
\phantomsection\label{pro:reductionToUniqueMaximalNode}
Let $G$ be a labeled graph. Then there exists a graph $G\p$ such that
\begin{enumerate}
\item $G$ is a subgraph of $G\p$,
\item $G\p$ has a unique node of maximal label that is reachable from all nodes of $G\p$ and,
\item the standard decompositions of $G$ and $G\p$ are in canonical
  bijection.
\end{enumerate}
\end{pro}
\begin{proof}
Let $l$ be the maximal label among all nodes of $G$. Let $G\p$ be
equal to $G$, except that $G\p$ has an extra node $v$ with label $l+1$,
and $v$ has an edge to it from all other nodes. The first two
conditions are immediate, so it remains to show that the standard
decompositions of $G$ and $G\p$ are in canonical bijection.

It is not hard to see that the function 
\begin{equation*}
  \begin{split}
    f : 
    \left\{ \text{standard cmps of } G\p \right\} \setminus \set{\set{v}}
    & \to 
    \left\{ \text{standard cmps of } G \right\} \\
    H & \mapsto H\setminus\set{v} \text{ without edges in } H \text{ with target } v
  \end{split}
\end{equation*}
is a bijection. We extend $f$ to
a map of multisets of standard components by applying it to each standard
component individually, so for example $f(\set{A,B}) \defeq \set{f(A),f(B)}$.

Let $D\p$ be a standard decomposition of $G\p$. Write $D\p$ as a union
of $D$ and $V$ where $V$ is a multiset that contains only
copies of $\set{v}$ while $D$ does not contain $\set v$ at all. 
Then obviously $f(D)$ is a decomposition of $G$.

For the other direction, let $D$ be a standard decomposition of $G$
and let $D\p\defeq f^{-1}(D)$. Then $G\p-\sum D\p$ is a graph in which
all nodes but $v$ have label zero, node $v$ having a label $l>0$. Let
$V$ be the multiset that contains $l$ copies of $\set v$. Then
$D\p\cup V$ is a decomposition of $G\p$. Let $g$ be the function
$D\mapsto D\p\cup V$. It is not hard to see that $f$ and $g$ are
mutual inverses.
\end{proof}

We can now prove that C4 decomposition and standard decomposition of
labeled graphs are equivalent.

\begin{proof}[Proof of Theorem \ref{thm:graphsVsC4}]
  (i) A solution of problem (a) implies a solution of problem (b) by Proposition \ref{pro:C4vsGraph}. 
  Assume we are able to solve problem (b), and are given a labeled graph $G$. 
  We pass to the canonicalization $G\p$, 
  which has the same standard decompositions as $G$ by Proposition \ref{pro:canonicalizatedGraph}. 
  If $G\p$ has multiple nodes of locally maximal label $l$, 
  we pass to the graph $G\pp$ with only one node of maximal label $l+1$
  from Proposition \ref{pro:reductionToUniqueMaximalNode}. 
  $G\pp$ still has the same standard decompositions as $G$. 
  Then we replace $G\pp$ by its transitive closure $G\ppp$. 
  By Lemma \ref{lmm:transitiveClosure}, this transition does not harm the decompositions either. 
  Finally, Proposition \ref{pro:graphToStdSet} provides a standard set $\Delta\ppp$
  whose canonicalized standard graph is $G\ppp$. 
  Problem (a) is solved. 
  
  (ii) This assertion depends on the representations of $G$ and $\Delta$. 
  In the proof of Theorem \ref{thm:algorithm}, we explained that we specify a graph as a list of nodes with labels and a list of edges. 
  After the proof of Proposition \ref{pro:graphToStdSet}, 
  we explained that we specify a standard set by its outer corners. 
  
  Let us first show that for any graph $G$ with $n$ nodes and $e$ edges, 
  a staircase $\Delta$ whose graph equals $G$ can be computed in polynomial time. 
  We may assume that $G$ is canonical and transitive, and has only one node of maximal label, since the operations 
  \begin{itemize}
    \item passing to the canonicalization, 
    \item passing to a graph with only one node of maximal label, and
    \item passing to the transitive closure
  \end{itemize}
  are obviously polynomial in the datum of $G$. 
  It therefore remains to show that the construction from the proof of Proposition \ref{pro:graphToStdSet} is polynomial. 
  That construction builds $\Delta$ using two nested inductions over $n$ and $e$. 
  The respective base cases being trivial, it suffices to show that both induction steps are polynomial in the datum of $G$. 
  Let us stick to the notation from the proof of Proposition \ref{pro:graphToStdSet}. 
  In addition to that notation, we define $\C_i \subseteq \N^d$ as the set of corners of $\Delta_i$ for $i = 0,1,2$. 
  In both the inner and the outer induction, the dimension of the standard sets involved rises by one. 
  Thus the dimension $d$ is polynomial in the datum of $G$. 
  The outer induction step is the passage from $\Delta_0$ to $\Delta_1$, as defined in \eqref{eqn:Delta0ToDelta1}. 
  That definition shows that $e_1 \in \C_0$ and  
  \[
    \C_1 = \left(\C_0 \setminus \set{e_1} \right) \cup \setBuilderBarRight{e_1 + e_i}{i = 1, \ldots, d-1} \cup \set{h_0 e_d} ,
  \]
  cf. Figure \ref{figure:pilar}. The inner induction step is thus polynomial. 
  
  The inner induction step is the passage from $\Delta_1$ via $\Delta_{1 \frac{1}{2}}$ to $\Delta_2$. 
  Remember that for $i = 0,1$, the node $v_i \in G_1$ corresponds to a connected component $C_i$ of $(\Delta_1)^{h_i}$. 
  Let $\C\p$ be the union of the following three sets: 
  \begin{itemize}
  \item all corners $\alpha \in \C_1$ such that $\alpha - e_j \in (q^d)^{-1}(C_1) \cap \Delta_1$ for some $e_j \neq e_1$, 
  \item the projections to the hyperplane $\{x_d = 0\}$ of all corners $\alpha \in \C_1$ such that $\alpha - e_j \in (q^d)^{-1}(C_1) \cap \Delta_1$ for some $e_j \neq e_1,e_d$, and 
  \item the elements $2e_1$ and $h_1 e_d$. 
  \end{itemize}
  Then $\C\p$ is the set of corners of $M_{1\frac{1}{2}}$ from \eqref{eqn:Delta1ToDelta3/2}. 
  Remember that $\Delta_{1\frac{1}{2}}$ is the union of $\Delta_1$ and $M_{1\frac{1}{2}}$. 
  The set $\C_{1\frac{1}{2}}$ of corners of $\Delta_{1\frac{1}{2}}$ is therefore obtained by
  \begin{itemize}
    \item collecting the exponents of least common multiples of $x^\alpha$ $x^\beta$, 
    for all $\alpha \in \C_1$ and all $\beta \in \C\p$, 
    \item and subsequently cleaning that set up, that is, 
    detecting pairs $\alpha,\beta$ such that $\alpha \in \beta + \N^d$ and deleting each such $\alpha$. 
  \end{itemize}
  This establishes the passage from $\Delta_1$ to $\Delta_{1 \frac{1}{2}}$. 
  As for the passage from $\Delta_{1 \frac{1}{2}}$ to $\Delta_2$, we construct a set of corners $\C\pp$ in an analogous way as 
  we constructed $\C\p$ in the three bulleted items above, but using $C_0$ rather than $C_1$ and $h_0$ rather than $h_1$. 
  Then $\Delta_2$ is the union of $\Delta_{1\frac{1}{2}}$ and the standard set with corners $\C\pp$. 
  The set $\C_2$ is therefore obtained from sets $\C_{1\frac{1}{2}}$ and $\C\pp$ 
  by the method of taking least common multiples and cleaning up which we employed above. 
  All operations are polynomial. 

  Let us now show that for each standard set $\Delta$, its canonicalized graph $G\p(\Delta)$ can be computed 
  in polynomial time. 
  In other words, we have to compute the connected components of the iysohypses in polynomial time. 
  We assume $\Delta$ to be given by its corner set $\C$. 
  For each $\alpha \in \C$, we define $\Delta_\alpha \defeq q^d(\alpha) + \oplus_{i=1}^{d-1}\N e_i$. 
  For each height $a$, we define $\Delta_a$ as the union of all $\Delta_\alpha$, for all $\alpha$ with $\card{\alpha} \leq a$. 
  Then the $a$-th isohypse is 
  \[
    \Delta^a = \Delta_a\setminus \Delta_{a-1} = \cup_{\card{\alpha} = a} (\Delta_\alpha \setminus T_{a-1}) . 
  \]
  Obviously each $E_\alpha \defeq \Delta_{\alpha} \setminus T_{a-1}$ is connected. 
  Moreover, it is easy to see that $E_\alpha \cup E_\beta$ is connected 
  if, and only if, the least common multiple of the monomials $x^{q^d(\alpha)}$ and $x^{q^d(\beta)}$ has its exponent outside of $T_{a-1}$. 
  Upon applying this observation to all $\alpha,\beta$ of total degree $a$, we compute the connected components of the $a$-th isohypse in polynomial time. 
\end{proof}


\appendix

\section{A generating function}
\phantomsection\label{app:generatingFunction}

We will now present a natural generating function
for the number of standard decompositions of a standard graph $G$. 
The analogue of this generating function in the setting of standard sets is discussed in \cite[Section 2.3]{components}. 

It is good to temporarily forget about labelings. So let $F$ be an unlabeled directed graph. 
Let $\mathcal{E}$ be the set of all standard 0-1 subgraphs of $F$\footnote{We
defined standard 0-1 subgraphs only for labeled graphs; if $F$ is unlabeled, we give each node the trivial label 1; 
then the notion of 0-1 subgraphs is well-defined.}
with node set $\nodes F$. 
We identify each $E \in \mathcal{E}$ with the \emph{characteristic function} of the labeling, that is, 
with the vector $\chi_E \defeq (\chi_{E,v})_{v \in \nodes F}$ indexed by nodes of $F$, with entries
\[
  \chi_{E,v} \defeq 
  \begin{cases}
    1 & \text{if } v \in \nodes E \\
    0 & \text{else}. 
  \end{cases}
\]
We define $\chi \defeq (\chi_{E,v})_{E \in \mathcal{E}, v \in \nodes F}$ 
to be the matrix whose rows are indexed by $\mathcal{E}$, 
the row with index $E$ being the vector $\chi_E$. 
Moreover, we introduce a vector $t \defeq (t_v)_{v \in \nodes F}$ of indeterminates, also indexed by nodes of $F$. 
If $w \defeq (w_v)_{v \in \nodes F}$ is any vector of nonnegative integers, indexed by nodes of $F$, 
we write $t^w \defeq \prod_{v \in \nodes F}t_v^{w_v}$. 
Consider the power series
\[
  g \defeq \prod_{E \in \mathcal{E}}\frac{1}{1 - t^{\chi_E}}. 
\]
We define integers $\Phi_\chi(w)$, one for each integer-valued vector $w$ as above, 
by expanding the power series $g$, 
\[
  g \eqdef \sum_{v \in \N^{\nodes F}}\Phi_\chi(w)t^w .
\]
$\Phi_\chi$ is called a \emph{vector partition function}, see \cite{sturmfelsVector}. 
Note that labelings of graphs $G$ with the same nodes and edges as $F$ correspond to vectors $w$ as above via 
\[
  w = (\lab G v)_{v \in \nodes F}.
\]
We denote by $G_w$ the labeled graph $G$ with the same nodes and edges as $F$
and labeling given by $w$. 

\begin{pro}
\begin{enumerate}
  \item Given any vector $w \in \N^{\nodes F}$, 
  the coefficient $\Phi_\chi(w)$ vanishes unless the labeled graph $G_w$ is standard. 
  \item If the labeled graph $G_w$ is standard, 
  the coefficient $\Phi_\chi(w)$ equals the number of standard decompositions of $F$. 
\end{enumerate}
\end{pro}

\begin{proof}
  We expand each term $\frac{1}{1 - t^{\chi_E}}$ in the product expression of $g$ as a geometric series,
  \[
    g = \prod_{E \in \mathcal{E}}(1 + t^{\chi_E} + t^{2\cdot\chi_E} + t^{3\cdot\chi_E}  + t^{4\cdot\chi_E}+ \ldots ).
  \]
  Upon expanding the product, we see that each monomial appearing in the series takes the shape 
  $m= \prod_{E \in \mathcal{F}}t^{n_E\cdot\chi_E}$ for some finite $\mathcal{F} \subseteq \mathcal{E}$
  and some $n_E \in \N$. 
  We replace the set $\mathcal{F}$ by the multiset $\hsym$ in which each $E \in \mathcal{F}$ 
  appears $n_E$ times. Since each member of $\hsym$ is a standard 0-1 subgraph of $F$, 
  the graph $G \defeq \sum\hsym$ standard graph and has the same nodes and edges as $F$. 
  The above monomial $m$ equals $\prod_{v \in \nodes F}t^{\lab G v}$. 
  This establishes (1). 
  
  As for (2), let $G$ be a standard graph with the same nodes and edges as $\nodes F$. 
  The above discussion shows that the coefficient of the monomial 
  $m \defeq \prod_{v \in \nodes F}t^{\lab G v}$ shows up in the expansion of $g$, 
  and its coefficient counts the number of ways of writing $G$ as a sum $G = \sum\hsym$
  of elements of $\mathcal{E}$. 
  This is just the number of standard decompositions of $G$. 
\end{proof}


\section{Partitions of partitions}
\phantomsection\label{app:partitionsOfPartitions}

Appendix \ref{app:generatingFunction} suggests a connection between standard decompositions and partitions. 
Let us further investigate this. 

\begin{ex}
  \begin{itemize}
    \item The set of partitions of an integer $n$ is in natural bijection with the set of standard sets of cardinality $n$ 
    by identifying a partition and its Young diagram (in the French notation). 
    \item If $p = \set{n_1, \ldots, n_h}$ (a multiset) is a partition of $n$ and for each $i$, $p_i$ is a partition of $n_i$, 
    we call $\set{p_1, \ldots, p_h}$ a \emph{partition of partition of $n$}. 
    The set of partitions of partitions of $n$ is in natural bijection with the set of standard sets $\Delta \subseteq \N^3$ 
    of cardinality $n$, \emph{together with all their C4 decompositions}. 
  \end{itemize}
\end{ex}

Both bijections are visualized in Figure \ref{figure:partitionsOfPartitions}. 
For generalizing the statements, we introduce the notion of \emph{C4 games}. 

\begin{center}
\begin{figure}[ht]
  \begin{picture}(330,140)
    \put(0,50){\small $\set{5,3,2,2} = $}    
    \put(60,33.4){\line(1,0){75}}
    \put(60,48.4){\line(1,0){75}}
    \put(60,63.4){\line(1,0){45}}
    \put(60,78.4){\line(1,0){30}}
    \put(60,93.4){\line(1,0){30}}
    \put(60,33.4){\line(0,1){60}}
    \put(75,33.4){\line(0,1){60}}
    \put(90,33.4){\line(0,1){60}}
    \put(105,33.4){\line(0,1){30}}
    \multiput(120,33.4)(15,0){2}{\line(0,1){15}}
    \put(150,51){\small 
    $\left\{ \begin{array}{c}
    \set{4,3} \\ \set{2,1} \\ \set{5}
    \end{array}
    \right\}
     = $}    
    \put(275,20){\line(3,2){10}}
    \put(260,35){\line(3,2){20}}
    \put(275,35){\line(3,2){20}}
    \put(245,35){\line(3,2){20}}
    \put(230,35){\line(3,2){20}}
    \put(275,20){\line(0,1){15}}
    \put(260,20){\line(0,1){15}}
    \put(245,20){\line(0,1){15}}
    \put(230,20){\line(0,1){15}}
    \put(230,20){\line(1,0){45}}
    \put(230,35){\line(1,0){45}}
    \put(285,26.7){\line(1,0){15}}
    \put(285,26.7){\line(0,1){15}}
    \put(300,41.7){\line(3,2){10}}
    \put(300,26.7){\line(3,2){10}}
    \put(300,26.7){\line(0,1){15}}
    \put(310,33.4){\line(0,1){15}}
    \put(240,41.7){\line(1,0){60}}
    \put(250,48.4){\line(1,0){60}}
    \put(220,40){\small $+$}
    \put(230,60){\line(1,0){15}}
    \put(230,75){\line(1,0){15}}
    \put(250,88.4){\line(1,0){30}}
    \put(230,60){\line(0,1){15}}
    \put(245,60){\line(0,1){15}}
    \put(245,60){\line(3,2){10}}
    \put(245,75){\line(3,2){20}}
    \put(230,75){\line(3,2){20}}
    \put(280,73.4){\line(0,1){15}}
    \put(255,66.7){\line(0,1){15}}
    \put(255,66.7){\line(1,0){15}}
    \put(240,81.7){\line(1,0){30}}
    \put(270,66.7){\line(3,2){10}}
    \put(270,66.7){\line(0,1){15}}
    \put(270,81.7){\line(3,2){10}}
    \put(220,80){\small $+$}
    \put(250,128.4){\line(1,0){75}}
    \put(240,121.7){\line(1,0){75}}
    \put(240,106.7){\line(1,0){75}}
    \multiput(240,106.7)(15,0){6}{\line(0,1){15}}
    \multiput(240,121.7)(15,0){6}{\line(3,2){10}}
    \put(315,106.7){\line(3,2){10}}
    \put(325,113.4){\line(0,1){15}}
  \end{picture}
\caption{Partitions (of partitions, resp.) and C4 games in $\N^2$ (in $\N^3$, resp.) correspond to each other}
\phantomsection\label{figure:partitionsOfPartitions}
\end{figure}
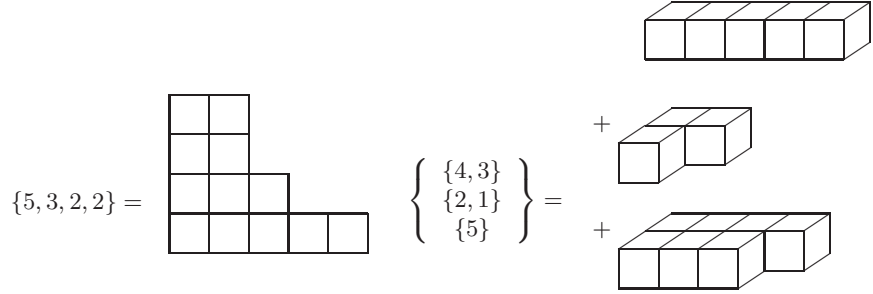
\end{center}

\begin{dfn}[Iterated partition]
  Let $n$ be a positive integer. We recursively define a \emph{$q$-fold iterated partition of $n$} as follows: 
  \begin{itemize}
    \item for $q=1$, it is a partition of $n$, that is, 
    a multiset $p = \set{n_1, \ldots, n_h}$ of positive integers such that $\sum n_i = n$;
    \item for $q>1$, it is a multiset $p = \set{p_1, \ldots, p_h}$ 
    of $(q-1)$-fold iterated partitions of integers $n_1, \ldots, n_h$ such that $\sum n_i = n$. 
  \end{itemize}
\end{dfn}

In other words, we look at all partitions of $n$ into $n_i$, together with all partitions of all parts $n_i$ into $n_{i,j}$, 
together with all partitions of all parts $n_{i,j}$ into $n_{i,j,k}$, etc. 

\begin{dfn}[C4 game]
  Let $n$ be a positive integer. We recursively define a \emph{C4 game of size $n$ in $\N^d$} as follows: 
  \begin{itemize}
    \item for $d=1,2$, it is standard set $\Delta \subseteq \N^d$ of cardinality $n$; 
    \item for $d>2$, it is a multiset $\set{g_1, \ldots, g_h}$ 
    of C4 games of respective sizes $n_i$ in $\N^{d-1}$ such that $\sum n_i = n$. 
  \end{itemize}
\end{dfn}

In other words, we look at all standard sets $\Delta \subseteq \N^d$ of a cardinality $n$, 
together with all C4 decompositions of $\Delta$ into $\Delta_i \subseteq \N^{d-1}$, 
together with all C4 decompositions of all $\Delta_i$ into $\Delta_{i,j} \subseteq \N^{d-2}$, 
together with all C4 decompositions of all $\Delta_{i,j}$ into $\Delta_{i,j,k} \subseteq \N^{d-3}$, etc. 

\begin{pro}
  For all $d, n \in \N$, there is a natural bijection
  \[
    f_d:\set{ (d-1)\text{-fold iterated partitions of } n }
    \to
    \set{ \text{C4 games of size }n \text{ in }\N^d }. 
  \]
\end{pro}

\begin{proof}
  The assertion is obvious for $d=1,2$. 
  For $d>2$, the bijection $f_d$ sends each multiset $\set{H_1, \ldots, H_l}$ 
  of $(d-2)$-fold iterated partitions of integers $n_1, \ldots, n_l$ to the multiset $\set{f_{d-1}(H_1), \ldots, f_{d-1}(H_l)}$. 
\end{proof}

Note that the bijection is only natural up to the choice of coordinate axes in $\N^d$. 
In other words, replacing the tuple $(e_1, \ldots, e_d)$ of standard basis elements by 
$(e_{\sigma(1)}, \ldots, e_{\sigma(d)})$ for some permutation $\sigma$ induces an automorphism of the 
source of bijection $f_d$. 
For $d=2$, this corresponds to the ambiguity between a partition and its transpose. 

\bibliography{references}

\providecommand{\bysame}{\leavevmode\hbox to3em{\hrulefill}\thinspace}
\providecommand{\MR}{\relax\ifhmode\unskip\space\fi MR }
\providecommand{\MRhref}[2]{%
  \href{http://www.ams.org/mathscinet-getitem?mr=#1}{#2}
}
\providecommand{\href}[2]{#2}
\begin{thebibliography}{Nak99}

\bibitem[Eis95]{eisenbud}
David Eisenbud, \emph{Commutative algebra}, Graduate Texts in Mathematics, vol.
  150, Springer-Verlag, New York, 1995, With a view toward algebraic geometry.
  \MR{MR1322960 (97a:13001)}

\bibitem[Eva05]{evain}
Laurent Evain, \emph{On the postulation of {$s\sp d$} fat points in {$\Bbb P\sp
  d$}}, J. Algebra \textbf{285} (2005), no.~2, 516--530. \MR{2125451
  (2005j:13019)}

\bibitem[Hir85]{hirschowitz}
Andr{\'e} Hirschowitz, \emph{La m\'ethode d'{H}orace pour l'interpolation \`a
  plusieurs variables}, Manuscripta Math. \textbf{50} (1985), 337--388.
  \MR{784148 (86j:14013)}

\bibitem[Led]{components}
Mathias Lederer, \emph{Components of {G}r\"obner strata in the {H}ilbert scheme
  of points}, Proc. London Math. Soc., to appear.

\bibitem[Led08]{jpaa}
M.~Lederer, \emph{The vanishing ideal of a finite set of closed points in
  affine space}, J. Pure Appl. Algebra \textbf{212} (2008), 1116--1133.

\bibitem[Nak99]{nakajima}
Hiraku Nakajima, \emph{Lectures on {H}ilbert schemes of points on surfaces},
  University Lecture Series, vol.~18, American Mathematical Society,
  Providence, RI, 1999. \MR{1711344 (2001b:14007)}

\bibitem[Stu95]{sturmfelsVector}
Bernd Sturmfels, \emph{On vector partition functions}, J. Combin. Theory Ser. A
  \textbf{72} (1995), no.~2, 302--309. \MR{1357776 (97b:52014)}

\end{thebibliography}
\bibliographystyle{amsalpha}

\end{document}